\documentclass[twoside,a4paper,11pt]{article}
\usepackage{amsfonts, amsbsy, amsmath, amsthm, amssymb, latexsym}
\usepackage{mathrsfs}
\usepackage{enumerate,centernot}
\usepackage[top=30mm,right=30mm,bottom=30mm,left=30mm]{geometry}

\usepackage{bm}
\usepackage{fancyhdr}

\numberwithin{equation}{section}

\renewcommand{\bf}{\textbf}

\newcommand{\leqs}{\leqslant}

\newtheorem{theorem}{Theorem}

\newtheorem{propn}[theorem]{Proposition}

\newtheorem{thm}{Theorem}[section]

\newtheorem{prop}[thm]{Proposition}
\newtheorem{lem}[thm]{Lemma}
\newtheorem{corol}[thm]{Corollary}

\newtheorem{con}[theorem]{Conjecture}

\theoremstyle{definition}
\newtheorem{defn}[thm]{Definition}

\begin{document}

\title{On finite simple images of triangle groups}
\author{Sebastian Jambor, Alastair Litterick, Claude Marion}

\date{\today}
\maketitle

\noindent\textbf{Abstract.}  For a simple algebraic group $G$ in characteristic $p$, 
a triple $(a,b,c)$ of positive integers is said to be rigid for $G$ if the dimensions of the subvarieties of $G$ of elements of order dividing $a,b,c$ sum to $2\dim G$. In this paper we complete the proof of a conjecture of the third author, that for a rigid triple $(a,b,c)$ for $G$ with $p>0$, the triangle group $T_{a,b,c}$ has only finitely many  simple images of the form  $G(p^r)$. We also obtain further results on the more general form of the conjecture, where the images $G(p^r)$ can be arbitrary  quasisimple groups of type $G$.

\section{Introduction}\label{s:introduction}
It has been known for some time \cite[Theorem B]{AB} that every finite simple group can be generated by two elements. 
It is natural to ask, therefore, whether a given finite simple group $G_0$ can  be generated by two elements, whose orders respectively divide given integers $a$ and $b$, and whose product has order dividing a given integer $c$.  A finite group generated by two such elements is called an $(a,b,c)$-group or said to be $(a,b,c)$-generated. Equivalently, an $(a,b,c)$-group is a finite quotient of the triangle group $T=T_{a,b,c}$ with presentation
$$ T=T_{a,b,c}=\langle x,y,z:x^a=y^b=z^c=xyz=1\rangle. $$
When investigating the finite (nonabelian) quasisimple quotients of $T$, we can  assume that $1/a+1/b+1/c<1$ as otherwise $T$ is either soluble or $T\cong T_{2,3,5}\cong{\rm Alt}_5$ (see \cite{Conder}). The group $T$ is then a hyperbolic triangle group.  Without loss of generality, we will further assume that $a \leq b \leq c$ and call $(a,b,c)$ a hyperbolic triple of integers. (Indeed, $T_{a,b,c}\cong T_{a',b',c'}$ for any permutation $(a',b',c')$ of $(a,b,c)$.)\\

 In 1893, Hurwitz \cite{Hurwitz} showed that the group ${\rm Aut}(S)$ of automorphisms of a compact Riemann surface 
$S$ of genus $h\geq 2$ is finite and has order bounded above by $84(h-1)$. Moreover the latter bound is attained if and only if ${\rm Aut}(S)$ is a $(2,3,7)$-group.  
Following the original work of Hurwitz, the first $(a,b,c)$-groups to be investigated  were the $(2,3,7)$-groups or so called Hurwitz groups. For example, Macbeath \cite{Macbeath} determined the prime powers $q$ such that ${\rm PSL}_2(q)$ is Hurwitz.  As another illustration, Conder showed in 1980 \cite{Conder1980} that every alternating group ${\rm Alt}_n$ of degree greater than 167 is Hurwitz.  Moreover in \cite{Conder1.5} he determined the positive integers $n\leq 167$ for which ${\rm Alt}_n$ is Hurwitz, in particular ${\rm Alt}_{167}$ is not $(2,3,7)$-generated and the latter 167 bound is best possible. 
Although a lot of effort has been put into determining the finite simple Hurwitz groups, this classification is not complete. For a survey of results, see for example \cite{Conder,Conder2}.  \\ Until recently, there were much fewer results in the literature for $(a,b,c)\neq (2,3,7)$. 
By Higman's conjecture, proved to hold separately in \cite{Everitt, LSalt}, every  Fuchsian group, and so  also every hyperbolic $T_{a,b,c}$, surjects onto all but finitely many alternating groups. \\
In \cite{Marionconj} the third author introduced some new methods, in particular the concept of rigidity, for investigating whether  a finite quasisimple group of Lie type is an $(a,b,c)$-group.\\

Let $G$ be a simple algebraic group defined over an algebraically closed field $K$ of  characteristic ${\rm char}(K)=p$ (possibly equal to 0).   For  a positive integer $u$, we let $G_{[u]}$ be the subvariety of $G$ consisting of elements of $G$ of order dividing  $u$ and set $j_u(G)= \dim G_{[u]}$. \\
Recall that a finite quasisimple group $G_0$ of Lie type  occurs as the derived subgroup of the  fixed point group  of a simple algebraic group $G$, when ${\rm char}(K)=p$ is  prime, under a Steinberg endomorphism  $F$, i.e. $G_0=(G^F)'$.    We use the standard notation $G_0=(G^F)'= G(q)$ where $q=p^r$ for some positive integer $r$. (We include the possibility that $G(q)$ is of twisted type.)\\
 Following \cite{Marionconj}, given a simple algebraic group $G$ and a triple $(a,b,c)$ of positive integers, we say that $(a,b,c)$ is rigid for $G$ if the sum $j_a(G)+j_b(G)+j_c(G)$ is equal to $2\dim G$. When the  latter sum is less (respectively, greater) than $2\dim G$, we say that $(a,b,c)$ is reducible (respectively, nonrigid)
 for $G$. \\
The main purpose of this paper is to further investigate the rigid case and make further progress on the finiteness conjecture of the third author  (first formulated for triples $(a,b,c)$ of primes in \cite{Marionconj}):

\begin{con}\label{c:marionconj} 
Let $G$ be a simple algebraic group defined over an algebraically closed field $K$ of prime characteristic $p$ and let $(a,b,c)$ be a rigid triple of integers for $G$.  Then there are only finitely many  quasisimple groups $G(p^r)$ (of the form $(G^F)'$ where $F$ is a Steinberg endomorphism) that are $(a,b,c)$-generated.  
\end{con}

In  \cite{Marionconj} the third author made some progress on Conjecture \ref{c:marionconj} in the special case where $a$, $b$ and $c$ are prime numbers. Given a simple algebraic group $G$ he determined in \cite[Theorem 3]{Marionconj} the rigid triples $(a,b,c)$ of primes that are rigid for $G$. The classification of rigid triples of primes for simple algebraic groups reduces Conjecture \ref{c:marionconj} (for triples of primes) to a handful of cases which are easy to describe. Moreover some of these cases can be handled using the concept of linear rigidity for a triple of elements in a general linear group  defined over $K$.   Following \cite[Theorems 1-3]{Marionconj} Conjecture \ref{c:marionconj}  holds in many cases when $a$, $b$ and $c$ are primes:  for example for $G$ not of type $C_\ell$ with $2\leq \ell \leq 13$ nor $G_2$. \\
Using deformation theory, Larsen, Lubotzky and the third author proved in \cite[Theorem 1.7]{LLM1} that Conjecture \ref{c:marionconj} holds except possibly if $p$ divides $abcd$ where $d$ is the determinant of the Cartan matrix of $G$ (see Theorem \ref{t:conmainllm} below). \\
In the proof of \cite[Theorem 1.7]{LLM1}, given a simple algebraic group $G$   and  a rigid triple $(a,b,c)$ of integers for $G$, one shows that if $p\nmid abcd$ then a surjective homomorphism $\rho: T\rightarrow G(p^r)$, seen as an element of the representation variety ${ \rm Hom}(T,G)$  of  homomorphisms from $T$ to $G$, is locally rigid. That is, there is a neighbourhood of $\rho$ in which every element  is obtained from $\rho$ by conjugation by an element of $G$.  In particular, the orbit of $\rho$ under the action of $G$ by conjugation is open and the conclusion of \cite[Theorem 1.7]{LLM1} follows from the fact that in a variety one can have only finitely many open orbits.  The assumption  in \cite[Theorem 1.7]{LLM1} that $p\nmid abcd$ is crucial. Indeed as $p\nmid abc$, there is a formula of Weil \cite{Weil} for computing the dimension of  the cohomology group $H^1(T, {\rm Ad}\circ \rho)$ where ${\rm Ad}: G \rightarrow {\rm Aut}({\rm Lie}(G))$ is the adjoint representation of $G$. Moreover, since $p\nmid d$ one can show that    $\dim H^1(T, {\rm Ad}\circ \rho)=0$ and then \cite{Weil} yields that $\rho$ is locally rigid as an element of ${\rm Hom}(T,G)$.\\
Although \cite[Theorem 1.7]{LLM1} is a major achievement in the study of Conjecture \ref{c:marionconj}, it leaves infinitely many open triples such as $(2,3,c)$ and $(3,3,c)$ for several families of classical groups.\\
 To make further progress on Conjecture  \ref{c:marionconj}, it  is essential to determine the rigid triples of integers (not necessarily primes) for simple algebraic groups.\\

 In this paper we examine the remaining open cases of the conjecture.  In our first result we complete the proof of the conjecture for the finite simple groups. 
 
  
  \begin{theorem}\label{t:marionconjsimple}
  Conjecture \ref{c:marionconj} holds for all finite simple groups of Lie type. That is, if $G$ is a simple algebraic group of adjoint type defined over an algebraically closed field $K$  of prime characteristic $p$ and $(a,b,c)$ is rigid for $G$, then  only finitely many simple groups $G(p^r)$ are $(a,b,c)$-groups. 
  \end{theorem}
  
  The first ingredient in the proof of Theorem \ref{t:marionconjsimple} is the classification of rigid triples $(a,b,c)$ of integers for simple algebraic groups. This is established in \cite{MarionLawther} and we recall the result in \S\ref{s:pre}.  In particular, if $G$ is of  adjoint type and $(a,b,c)$ is rigid for $G$ then  $G$ is of type ${\rm A}_\ell$ (with $1\leq \ell\leq 4$), $C_2$ or $G_2$. More precisely, the rigid triples $(a,b,c)$ of integers for simple algebraic groups $G$ of adjoint type are as in Table \ref{ta:rigida} below.

  \begin{table}[h!]
\center{
\begin{tabular}{|l|l|l|}
\hline
$G$ & $p$ & $(a,b,c)$\\
\hline
${\rm PSL}_2(K)$ & any & $(a,b,c)$\\
${\rm PSL}_3(K)$ & any  & $(2,b,c)$\\
${\rm PSL}_4(K)$ & any  & $(2,3,c)$\\
${\rm PSL}_5(K)$ & any  & $(2,3,c)$\\
\hline
${\rm PSp}_4(K)$ & any & $(2,3,c)$, $(3,3,c)$\\
\hline
$G_{2}(K)$ & any  & $(2,4,5)$, (2,5,5)\\
\hline
\end{tabular}}
\caption{Rigid triples for simple algebraic groups  of adjoint type}\label{ta:rigida}
\end{table}

Recall that given a field $\mathbb{F}$ and a subgroup $H$ of ${\rm GL}_n(\mathbb{F})$, the character field of $H$ is the subfield of $\mathbb{F}$ generated by the traces of the elements of $H$. If $G$  is of type $C_2$ and $(a,b)=(3,3)$, our approach is  to   determine some generators (as well as some relations between these generators) for the character fields of the subgroups of  ${\rm Sp}_4(K)$ generated by two elements of order 3 having product of order dividing a given positive integer $c$ and acting absolutely irreducibly on the natural module for ${\rm Sp}_4(K)$.  The characterisation we obtain for these character fields implies that up to isomorphism there are only finitely many such fields.  In particular, given a prime number $p$ and a positive integer $c$, there are only finitely many positive integers $r$ such that ${\rm Sp}_4(p^r)$ (respectively, ${\rm PSp}_4(p^r)$) is a $(3,3,c)$-group. \\
 The techniques  involve linear algebra, commutative algebra and the use of Gr\"{o}bner bases.  \\
  The case where $G$ is of type $A$ can be dealt with using the concept of linear rigidity originally introduced in \cite{Katz}.  The argument we give for $G$ of type $A$ is almost identical to the one given in \cite[Lemma 3.2]{Marionconj} for a triple $(a,b,c)$ of primes.\\
 In the remaining cases, namely where $G$ is of type $C_2$ and $(a,b)=(2,3)$ or $G$ is of type $G_2$ and $(a,c)=(2,5)$, we embed $G$ in a simple algebraic group $H$ of type $A$ and show that either an $(a,b,c)$-subgroup of $G$ acts reducibly on the natural module for $H$ and so $G(q)$ is never an $(a,b,c)$-group, or use linear rigidity to conclude. \\
  
More  generally, considering also non-adjoint groups, we obtain the following result. 

\begin{theorem}\label{t:cormain} Conjecture \ref{c:marionconj} holds except possibly if $G$, $p$ and $(a,b,c)$ are as in Table \ref{ta:marionconjex} below. 
\end{theorem}

\begin{table}
\begin{center}
\begin{tabular}{|l|l|l|}
\hline
$G$& $(a,b,c)$ & $p$ \\
\hline
${\rm Sp}_4(K)$ & $(2,b,c)$ $b\geq 5$ & $p\neq 2$ and $p\mid bc$\\ 
& $(3,4,c)$ $c\geq 5$ & $p\neq 2$, and $p=3$ or $p \mid c$\\
& $(4,4,c)$ $c\geq 5$ & $p\neq 2$  and $p\mid c$\\
\hline
${\rm Sp}_6(K)$ & $(2,5,c)$ $c\geq7$ & $p\neq 2$, and $p=5$ or $p \mid c$\\
&  $(2,6,c)$ $c\geq7$ & $p\neq 2$, and $p=3$ or $p \mid c$\\
& $(3,3,4)$ & $p=3$\\
& $(3,4,4)$ & $p=3$\\
\hline
${\rm Sp}_8(K)$ & $(2,3,c)$ $c\geq9$ & $p\neq 2$, and $p=3$ or $p \mid c$\\
&  $(2,4,7)$  & $p=7$\\
& $(2,5,5)$ & $p=5$\\
& $(2,5,6)$ & $p=3$ or $p=5$\\
& $(2,6,6)$ & $p=3$\\
\hline
${\rm Sp}_{10}(K)$ & $(2,3,c)$ $c\geq11$ & $p\neq 2$, and $p=3$ or $p \mid c$\\
&  $(2,4,7)$  & $p=7$\\
\hline
${\rm Sp}_{2\ell}(K)$, $\ell\in\{6,7\}$ & $(2,3,9)$ & $p=3$\\
& $(2,3,10)$ & $p=3$ or $p=5$\\
&  $(2,4,6)$  & $p=3$\\
\hline
${\rm Sp}_{2\ell}(K)$, $\ell\in\{8,9\}$ & $(2,3,8)$ & $p=3$\\
& $(2,4,5)$ & $p=5$\\
\hline
${\rm Sp}_{2\ell}(K)$, $\ell\in\{10,12,13\}$ & $(2,3,7)$ & $p=3$ or $p=7$\\
\hline
\end{tabular}
\end{center}
\caption{The open cases for Conjecture \ref{c:marionconj}}\label{ta:marionconjex}
\end{table}


\newpage

Once we have Theorem \ref{t:marionconjsimple} to hand, Theorem \ref{t:cormain} follows essentially from the classification determined in \cite{MarionLawther} of rigid triples of integers for non-adjoint simple algebraic groups and the fact that Conjecture \ref{c:marionconj} holds generally for $G$ non-adjoint  of types $A$, $B$ and $D$.  \\

Theorems \ref{t:marionconjsimple} and \ref{t:cormain} consider the finite quasisimple quotients  $(G^F)'=G(p^r)$ of Lie type of a given hyperbolic triangle group $T=T_{a,b,c}$ for which $(a,b,c)$ is  rigid for $G$. 
Before describing the outline of the paper, we give some recent complementary results on $(a,b,c)$-generation of finite groups of Lie type that essentially correspond to the cases where $(a,b,c)$ is not rigid for $G$.\\
  If $(a,b,c)$ is reducible for $G$, the third author proved the following non-existence result:
  
 \begin{propn}\label{p:marionred}\cite[Proposition 10]{MarionLawther}.
 Let $G$ be a simple algebraic group defined over an algebraically closed field $K$ of prime characteristic $p$ and let $(a,b,c)$ be a reducible triple of integers for $G$.  Then a quasisimple group $G(p^r)$ (of the form $(G^F)'$ where $F$ is a Steinberg endomorphism) is never an $(a,b,c)$-group.
\end{propn}

 \begin{table}[h]
\center{
\begin{tabular}{|l|l|l|}
\hline
$G$ & $p$ & $(a,b,c)$\\
\hline
${\rm SL}_2(K)$ & $p \neq 2$  & $(2,b,c)$\\
\hline
${\rm Sp}_4(K)$ & $p \neq 2$  & $(2,3,c)$, $(2,4,c)$, $(3,3,4)$, $(3,4,4)$, $(4,4,4)$\\
${\rm Sp}_6(K)$ & $p \neq 2$  & $(2,3,c)$, $(2,4,c)$, $(2,5,5)$, $(2,5,6)$, $(2,6,6)$\\
${\rm Sp}_8(K)$ & $p \neq 2$  & $(2,3,7)$, $(2,3,8)$, $(2,4,5)$, $(2,4,6)$\\
${\rm Sp}_{10}(K)$ & $p \neq 2$  & $(2,3,7)$, $(2,3,8)$, $(2,3,9)$, $(2,3,10)$, $(2,4,5)$, $(2,4,6)$\\
${\rm Sp}_{12}(K)$ & $p \neq 2$  & $(2,3,7)$, $(2,3,8)$, $(2,4,5)$\\
${\rm Sp}_{14}(K)$ & $p \neq 2$  & $(2,3,7)$, $(2,3,8)$, $(2,4,5)$\\
${\rm Sp}_{16}(K)$ & $p \neq 2$  & $(2,3,7)$\\
${\rm Sp}_{18}(K)$ & $p \neq 2$  & $(2,3,7)$\\
${\rm Sp}_{22}(K)$ & $p \neq 2$  & $(2,3,7)$\\
\hline
\end{tabular}
}
\caption{Reducible triples for simple algebraic groups of simply connected or adjoint type}\label{ta:red}
\end{table}

In \cite{MarionLawther} the third author classified the reducible triples $(a,b,c)$ of integers for simple algebraic groups $G$, see Table \ref{ta:red} below. In particular if $(a,b,c)$ is reducible for $G$ then $G$ is abstractly isomorphic to a simple algebraic group of simply connected type. Combining the above result with Table \ref{ta:red}, we obtain:

\begin{propn}
If $(G,p,(a,b,c))$ is as in  Table \ref{ta:red}, then $G(p^r)$ is never an $(a,b,c)$-group. 
\end{propn}
  
  We now make some observations on the nonrigid case.  Let us note that the converse of Conjecture \ref{c:marionconj} is not true in the sense that we can have a simple algebraic group $G$ defined over an algebraically closed field $K$ of prime characteristic $p$ and a nonrigid triple $(a,b,c)$ of integers  for $G$ such that there are only finitely many positive integers $r$ such that $G(p^r)$ is an $(a,b,c)$-group. For example by \cite{Tamburini} ${\rm SL}_7(p^r)$ is never a $(2,3,7)$-group although the triple $(2,3,7)$ is nonrigid for ${\rm SL}_7(K)$. However  Larsen, Lubotzky and the third author have  investigated through deformation theory the case where $G(p^r)$ is a finite untwisted simple group of Lie type and $(a,b,c)$ is a nonrigid hyperbolic triple of integers for $G$. To state the main result in this direction we introduce some notation and recall the notion for a hyperbolic triangle group $T$ to be saturated with finite untwisted simple quotients of type $X$. Let $X$ be a simple Dynkin diagram, $X(\mathbb{C})$ be the simple adjoint algebraic group over $\mathbb{C}$ and $X(p^r)$ denote the untwisted finite simple group of type $X$ over $\mathbb{F}_{p^r}$. A hyperbolic triangle group $T$ is said to be saturated with finite quotients of type $X$ if there exist integers $p_0$ and $e$ such that for all primes $p>p_0$, $X(p^{e\ell})$ is a quotient of $T$ for every $\ell \in \mathbb{N}$.    
The motivation for this notion stems from the property, proved in \cite{LLM2}, that a hyperbolic triangle group  $T$ is saturated with finite quotients of a given (Lie) type $X$ if and only if there exists a representation $\rho \in {\rm Hom}(T,X(\mathbb{C}))$ such that  ${\rm Im}(\rho)$ is Zariski-dense and $\dim H^1(T,{\rm Ad}\circ \rho)>0$ (where ${\rm Ad}: X(\mathbb{C})\rightarrow {\rm Aut}({\rm Lie}(X(\mathbb{C})))$ is the adjoint representation of $X(\mathbb{C})$).
  
  \begin{theorem}  \cite[Theorem 1.1]{LLM2}.\label{t:llm1}
  A hyperbolic triangle group $T_{a,b,c}$ is saturated with finite quotients of type $X$ except possibly if $(T,X)$ appears in \cite[Tables 1 or 2]{LLM2}. 
  \end{theorem}
   The cases appearing in \cite[Table 2]{LLM2} are in fact  true exceptions to the theorem as they correspond to the rigid cases. It is believed that most of the other  possible exceptions appearing in \cite[Table 1]{LLM2} are  not exceptions to the theorem. It   would be interesting to further investigate this theorem  and  for example eliminate some possible exceptions to the theorem appearing in \cite[Table 1]{LLM2}. \\


The outline of the paper is as follows. In \S\ref{s:pre} we present some preliminary results needed in the proof of Theorem \ref{t:marionconjsimple}. We recall in particular  from \cite{LLM1} that given  a simple algebraic group $G$ and a rigid triple $(a,b,c)$ of positive integers for $G$, Conjecture \ref{c:marionconj}  holds except for finitely many explicit prime characteristics (see Theorem \ref{t:conmainllm}). We also recall  the classification of rigid triples of integers for simple algebraic groups  \cite{MarionLawther} (see Theorem \ref{t:classification} below). This is the most important ingredient in this paper. We also obtain in \S\ref{s:pre} two reduction results to Theorems \ref{t:marionconjsimple} and \ref{t:cormain}. 
  In \S\ref{s:concepts} we prove that Theorems \ref{t:marionconjsimple} and \ref{t:cormain} hold provided that Conjecture \ref{c:marionconj} holds for $G$ of type $C_2$ and $(a,b)=(3,3)$.  Note that to handle the relevant cases where $G$ is of type $A$ we essentially repeat a couple of arguments  from \cite{Marionconj}. 
   In \S\ref{s:sp433} we prove that Conjecture \ref{c:marionconj} holds for $G$ of type $C_2$ and $(a,b)=(3,3)$, completing the proofs of Theorems \ref{t:marionconjsimple} and \ref{t:cormain}. 
  \\

\noindent \textbf{Acknowledgements.}  The third author thanks Nikolay Nikolov for his insight on how to tackle the conjecture at its early stage when it was first proposed.  The third author also thanks the  MARIE CURIE and PISCOPIA  research fellowship grant scheme at the University of Padova for support. The research leading to these results has received funding from the European Comission, Seventh Framework Programme (FP7/2007-2013) under Grant Agreement 600376.\\
The second author acknowledges support from DFG Sonderforschungsbereich 701  at Bielefeld University. \\
The work presented received  support from the Swiss National Science Foundation under the grant $200021-153543$.\\
The authors are grateful to Martin Liebeck  
for valuable comments.

\section{Rigid triples}\label{s:pre}

In this section we  present some preliminary results needed in the proof of Theorem \ref{t:marionconjsimple}. We denote by $G$ a simple algebraic group defined over an algebraically closed field $K$ of characteristic $p$ (possibly equal to 0).\\  
Recall that 
given a positive integer $u$, we let $G_{[u]}$ be the subvariety of $G$ consisting of elements  of $G$ of order dividing $u$ and set $j_u(G)=\dim G_{[u]}$.  We  also set  $d_u(G)={\rm codim} \ G_{[u]}$.
By \cite[Lemma 3.1]{Lawther} $d_u(G)$  is the minimal dimension of a centralizer in $G$ of an element of $G$ of order dividing $u$ and  so $j_u(G)$ is the maximal dimension of a conjugacy class of $G$ of an element of $G$ of order dividing $u$. \\

 
Let us record the main result towards Conjecture \ref{c:marionconj}. Recall that the determinant of the Cartan matrix of a simple algebraic group $G$ is $\ell+1$, 2, 2, 4, 1, 1, 3, 2 and 1 according respectively as $G$ is of type $A_\ell$, $B_\ell$, $C_\ell$, $D_\ell$, $G_2$, $F_4$, $E_6$, $E_7$ and $E_8$.

 \begin{thm}\label{t:conmainllm}\cite[Theorem 1.7]{LLM1}.
 Let $G$ be a simple algebraic group defined over an algebraically closed field $K$ of prime characteristic $p$. Let $(a,b,c)$ be a rigid triple of integers for $G$ and let $d$ be the determinant of the Cartan matrix of $G$.  If $p\nmid abcd$ then there are only finitely many  quasisimple groups $G(p^r)$ (of the form $(G^F)'$ where $F$ is a Steinberg endomorphism) that are $(a,b,c)$-generated.  
  \end{thm}

  In order to prove Theorems \ref{t:marionconjsimple} and \ref{t:cormain}, we also need the classification of  hyperbolic triples $(a,b,c)$ of integers as reducible, rigid or nonrigid for a given simple algebraic group $G$.  This classification is established in \cite[Theorem 9]{MarionLawther}.  For ease of reference we state the result below, with Tables \ref{ta:rigida}  and  \ref{ta:red} recorded in \S\ref{s:introduction} and Table  \ref{ta:rigidsc}  recorded below. Recall that if $G$ is not of simply connected type  nor of adjoint type then either $G$ is abstractly isomorphic to ${\rm SL}_n(K)/C$ where $C\leqs Z({\rm SL}_n(K))$, or $G$ is of type $D_{\ell}$, $p \neq 2$ and $G$ is abstractly isomorphic to ${\rm SO}_{2\ell}(K)$ or a half-spin group ${\rm HSpin}_{2\ell}(K)$ where $\ell$ is even in the latter case. \\
  \begin{thm}\cite[Theorem 9]{MarionLawther}.\label{t:classification}
  The following assertions hold.
  \begin{enumerate}[(i)]
  \item The reducible hyperbolic triples $(a,b,c)$ of integers for simple algebraic groups $G$ of simply connected or adjoint type  are exactly those given in Table \ref{ta:red}. In particular there are no reducible hyperbolic triples of integers for  simple algebraic groups of adjoint type.
  \item The rigid hyperbolic triples of integers $(a,b,c)$ for simple algebraic  groups $G$ of simply connected type are exactly those given in Table \ref{ta:rigidsc}.
    \item The rigid hyperbolic triples of integers $(a,b,c)$ for simple algebraic groups $G$ of adjoint type are exactly those given in Table \ref{ta:rigida}.
  \item The classification of reducible and rigid  hyperbolic triples of integers for ${\rm SO}_n(K)$ is the same as for ${\rm PSO}_n(K)$.
  \item The classification of reducible and rigid hyperbolic triples of integers for ${\rm HSpin}_{2\ell}(K)$ is the same as for ${\rm Spin}_{2\ell}(K)$. 
  \item If $C\leqs Z({\rm SL}_n(K))$ contains an involution then the classification of reducible and rigid hyperbolic triples of integers for ${\rm SL}_n(K)/C$ is the same as for ${\rm PSL}_n(K)$. Otherwise, the classification of reducible and rigid  hyperbolic triples of integers for ${\rm SL}_n(K)/C$ is the same as for   ${\rm SL}_n(K)$.
  \end{enumerate}
  \end{thm}

  \begin{table}[h!]
   \center{
\begin{tabular}{|l|l|l|}
\hline
$G$ & $p$ & $(a,b,c)$\\
\hline
${\rm SL}_2(K)$ & $p =2$  & $(a,b,c)$\\
& $p \neq 2$ & $(a,b,c)$ $a\geq 3$\\
${\rm SL}_3(K)$ & any  & $(2,b,c)$\\
${\rm SL}_4(K)$ & $p =2$  & $(2,3,c)$\\
& $p\neq 2$ & $(2,3,c)$, $(2,4,c)$, $(3,3,4)$, $(3,4,4)$,$(4,4,4)$\\
${\rm SL}_5(K)$ & any  & $(2,3,c)$\\
${\rm SL}_6(K)$ & $p\neq 2$  & $(2,3,c)$, $(2,4,5)$, $(2,4,6)$\\
${\rm SL}_{10}(K)$ & $p\neq 2$  & $(2,3,7)$\\
\hline
${\rm Sp}_4(K)$ & $p = 2$  & $(2,3,c)$, $(3,3,c)$\\
${\rm Sp}_4(K)$ & $p \neq 2$  & $(2,b,c)$ $b \geq 5$, $(3,3,c)$ $c \geq 5$, $(3,4,c)$ $c\geq 5$, $(4,4,c)$ $c \geq 5$\\ 
${\rm Sp}_6(K)$ & $p \neq 2$  & $(2,5,c)$ $c \geq 7$,  $(2,6,c)$ $c\geq 7$, $(3,3,4)$, $(3,4,4)$, $(4,4,4)$\\
${\rm Sp}_8(K)$ & $p \neq 2$  & $(2,3,c)$ $c\geq 9$, $(2,4,7)$, $(2,4,8)$, $(2,5,5)$, $(2,5,6)$, $(2,6,6)$\\
${\rm Sp}_{10}(K)$ & $p \neq 2$  & $(2,3,c)$ $c\geq 11$,  $(2,4,7)$, $(2,4,8)$\\
${\rm Sp}_{12}(K)$ & $p \neq 2$  & $(2,3,9)$, $(2,3,10)$, $(2,4,6)$\\
${\rm Sp}_{14}(K)$ & $p \neq 2$  & $(2,3,9)$, $(2,3,10)$, $(2,4,6)$\\
${\rm Sp}_{16}(K)$ & $p \neq 2$  & $(2,3,8)$, $(2,4,5)$\\
${\rm Sp}_{18}(K)$ & $p \neq 2$  & $(2,3,8)$, $(2,4,5)$\\
${\rm Sp}_{20}(K)$ & $p \neq 2$  & $(2,3,7)$\\
${\rm Sp}_{24}(K)$ & $p \neq 2$  & $(2,3,7)$\\
${\rm Sp}_{26}(K)$ & $p \neq 2$  & $(2,3,7)$\\
\hline
${\rm Spin}_{11}(K)$ & $p \neq 2$  & $(2,3,7)$\\
${\rm Spin}_{12}(K)$ & $p \neq 2$  & $(2,3,7)$\\
\hline
\end{tabular}
}
\caption{Rigid triples for simple algebraic groups  of simply connected type}\label{ta:rigidsc}
\end{table}
  
 We obtain the following result  as an immediate  corollary to Theorems \ref{t:conmainllm} and \ref{t:classification}.  

\begin{corol}\label{c:s2c1}
Theorem \ref{t:marionconjsimple} holds provided that it holds for $(G,p,(a,b,c))$ as given below. 
\begin{enumerate}[(i)]
\item $G={\rm PSL}_2(K)$ and $p\mid 2abc$.
\item $G={\rm PSL}_3(K)$, $a=2$ and $p\mid 6bc$.
\item $G={\rm PSL}_4(K)$, $(a,b)=(2,3)$ and $p\mid 6c$.
\item $G={\rm PSL}_5(K)$, $(a,b)=(2,3)$ and $p\mid 30c$.
\item $G={\rm PSp}_4(K)$, $(a,b)\in \{(2,3),(3,3)\}$ and $p\mid 6c$.
\item $G=G_2(K)$, $p\in \{2,5\}$ and $(a,b,c)\in \{(2,4,5),(2,5,5)\}$.
\end{enumerate}
\end{corol}

We also obtain the following reduction result for Theorem \ref{t:cormain}.

\begin{corol}\label{c:s2c2}
Theorem \ref{t:cormain} holds provided that
\begin{enumerate}[(i)]
\item Theorem \ref{t:marionconjsimple}  holds,
\item Conjecture \ref{c:marionconj} holds  for $G$ non-adjoint of type $A$, and
\item  Conjecture \ref{c:marionconj} holds for $G={\rm Sp}_4(K)$, $p\neq2$ and $(a,b)=(3,3)$.
\end{enumerate}
\end{corol}

 \begin{proof}
The result is again a consequence of   Theorems \ref{t:conmainllm} and \ref{t:classification}. Here we need the further observation that if ${\rm char}(K)=2$ then ${\rm Sp}_4(K)$ is abstractly isomorphic to ${\rm PSp}_4(K)$.  Moreover we refer to  \cite[Proposition 3.2]{Marionconj} where it is shown  that if $p\neq 2$ and $n\in \{11,12\}$ then  ${\rm Spin}_n(p^r)$ and ${\rm HSpin}_{12}(p^r)$  are never Hurwitz groups (i.e. $(2,3,7)$-groups).
\end{proof}

We complete this section by giving below the main ideas involved in the proof of Theorem \ref{t:classification}. We refer the reader to \cite{MarionLawther}  for more details.\\

\noindent \textit{Sketch of the proof of Theorem \ref{t:classification}.}
 The principal ingredient in the proof of  Theorem \ref{t:classification} is the determination of $d_u(G)$ for every simple algebraic group $G$ and every positive integer $u$. Indeed, given a simple algebraic group $G$ over an algebraically closed field $K$ and a hyperbolic triple $(a,b,c)$ of  integers, saying that $(a,b,c)$ is rigid (respectively reducible, nonrigid) for $G$ amounts to saying that the sum $\mathbf{S}_{(a,b,c)}=d_a(G)+d_b(G)+d_c(G)$ is equal to (respectively, greater than, smaller than) $\dim G$. The value for $d_u(G)$ depends on the Euclidean division of the Coxeter number $h$ of $G$ by $u$. In particular, if $u>h$ then $d_u(G)$ is equal to the Lie rank $\ell$ of $G$.  \\
To determine $d_u(G)$ one first shows that 
$$ d_u(G_{a.})\leq d_u(G)\leq d_u(G_{s.c})$$
where $G_{s.c.}$ (respectively, $G_{a.}$) denotes the simple algebraic group of simply connected (respectively, adjoint) type over $K$ of same Lie type and Lie rank as $G$.\\
One then uses the result of \cite{Lawther} determining $d_u(G_{a.})$ and giving a lower bound for $d_u(G)$. The next step is to determine $d_u(G)$ for $G$ a simple algebraic group of simply connected type. One proceeds by constructing a specific  element $g$ of $G_{s.c.}$ of order dividing $u$ with centralizer  $C_{G_{s.c.}}(g)$ of small dimension and  one then   proves that $\dim C_{G_{s.c.}}(g)$ must  in fact be equal to $d_u(G_{s.c.})$. 
For $G$ not abstractly isomorphic to a group of simply connected or adjoint type, one has to pursue further considerations (see the proof of \cite[Theorem 5]{MarionLawther}). \\
With the principal ingredient to hand, one can show that, given  a simple algebraic group $G$, $d_u(G): \mathbb{N}\rightarrow \mathbb{N}$ seen as a function of $u$ is decreasing.   \\
One then puts a partial order on the set of hyperbolic triples of integers in the following way. Given two hyperbolic triples $(a,b,c)$ and $(a',b',c')$, we say that $(a,b,c)\leq (a',b',c')$ if and only if $a\leq a'$, $b\leq b'$ and $c\leq c'$. In particular, among all hyperbolic triples, exactly three are minimal: $(2,3,7)$, $(2,4,5)$ and  $(3,3,4)$. Given a simple algebraic group $G$, since  $d_u(G)$ is a decreasing function of $u$, if $(a,b,c)$ is nonrigid for $G$ then $(a',b',c')$ is nonrigid for $G$ for every $(a',b',c')\geq (a,b,c)$.  If $G$ is of type $A_\ell$ with $\ell \geq 11$, or $G$ is of type $B_\ell$ with $\ell\geq 10$,  or $G$ is of type $C_\ell$ with $\ell\geq15$, or $G$ is of type $D_\ell$ with $\ell \geq 9$, one can show, using general tight upper bounds for $d_u(G)$, that $(2,3,7)$, $(2,4,5)$ and $(3,3,4)$ are all nonrigid for $G$. For $G$ of exceptional type or of classical type of low rank,  one uses the precise values for $d_u(G)$ to classify the hyperbolic triples of integers for $G$ as reducible, rigid or nonrigid. \\
If one only considers the case where $G$ is of adjoint type,  less work is involved. 
As an illustration, we consider the case where $G$ is of type $A_\ell$ or $E_8$.  Given $G$ and a  hyperbolic triple $(a,b,c)$ of integers, we let $\mathbf{D}_{(a,b,c)}=\mathbf{S}_{a,b,c}-\dim G$,  so that $(a,b,c)$ is rigid (respectively, reducible, nonrigid) for $G$ if $\textbf{D}_{(a,b,c)}$ is equal to 0 (respectively, positive, negative). Also given the Coxeter number $h$ of $G$ and a positive integer $u$, we write $h=zu+e$  where $z$ and $e$ are nonnegative integers such that $0\leq e\leq u-1$.\\
Suppose first that $G$ is of type $A_\ell$ so that $h=\ell+1$ and $\dim G=h^2-1$. By \cite[Theorem 1]{Lawther}
\begin{equation}\label{e:toutmys}
d_u(G)=z^2u+e(2z+1)-1  =   \frac{(h-e)(h+e)}u+e-1. 
\end{equation} 
In particular, one can show that $d_u(G)$ is a decreasing function of $u$.
Let $g(e)= \frac{(h-e)(h+e)}u+e-1$ Then $g'(e)=\frac{u-2e}u$, and $g'(e)>0$ if and only if $e<u/2$. Hence $d_u(G)\leq F(u)$ where 
\begin{eqnarray*}
F(u) &  =& g(u/2)\\
& = & \frac{u^2-4u+4h^2}{4u}.  
\end{eqnarray*}
In particular, for any hyperbolic triple $(a,b,c)$ of integers, we have $\mathbf{S}_{(a,b,c)}\leq F(a)+F(b)+F(c)$.\\
Suppose $(a,b,c)=(2,3,7)$. We have $\mathbf{S}_{(2,3,7)}\leq F(2)+F(3)+F(7)$. Now
$$F(2)=\frac{h^2-1}2, \quad F(3)=\frac{4h^2-3}{12}, \quad F(7)=\frac{4h^2+21}{28}.$$
Hence $$\mathbf{S}_{(2,3,7)}\leq  \frac{41h^2}{42}$$
and $$\mathbf{D}_{(2,3,7)}\leq -\frac{h^2}{42}+1.$$
Since   $-\frac{h^2}{42}+1$ is negative for $h \geq 7$, it follows that $(2,3,7)$ is nonrigid for $G$ provided $h\geq 7$, that is  $\ell\geq 6$. Also every hyperbolic triple $(a',b',c')$ of integers with $(a',b',c')\geq (2,3,7)$ is nonrigid for $G$ with $\ell\geq 6$.\\ 
Suppose $(a,b,c)=(2,4,5)$. We have $\mathbf{S}_{(2,4,5)}\leq F(2)+F(4)+F(5)$. Now
$$F(2)=\frac{h^2-1}2, \quad F(4)=\frac{h^2}{4}, \quad F(5)=\frac{4h^2+5}{20}.$$
Hence $$\mathbf{S}_{(2,4,5)}\leq  \frac{19h^2}{20}-\frac14$$
and $$\mathbf{D}_{(2,4,5)}\leq -\frac{h^2}{20}+\frac{3}4.$$
Since   $-\frac{h^2}{20}+\frac 34$ is negative for $h \geq 4$, it follows that $(2,4,5)$ is nonrigid for $G$ provided $h\geq 4$, that is  $\ell\geq 3$. Also every hyperbolic triple $(a',b',c')$ of integers with $(a',b',c')\geq (2,4,5)$ is nonrigid for $G$ with $\ell\geq 3$.\\
Suppose $(a,b,c)=(3,3,4)$. We have $\mathbf{S}_{(3,3,4)}\leq F(3)+F(3)+F(4)$. Now
$$ F(3)=\frac{4h^2-3}{12}, \quad F(4)=\frac{h^2}{4}.$$
Hence $$\mathbf{S}_{(3,3,4)}\leq  \frac{11h^2}{12}-\frac12$$
and $$\mathbf{D}_{(3,3,4)}\leq -\frac{h^2}{12}+\frac{1}2.$$
Since   $-\frac{h^2}{12}+\frac 12$ is negative for $h \geq 3$, it follows that $(3,3,4)$ is nonrigid for $G$ provided $h\geq 3$, that is  $\ell\geq 2$. Also every hyperbolic triple $(a',b',c')$ of integers with $(a',b',c')\geq (3,3,4)$ is nonrigid for $G$ with $\ell\geq 2$. \\
To classify the hyperbolic triples  $(a,b,c)$ of integers as reducible, rigid or nonrigid  for $G$ of type $A_\ell$ with $\ell \leq 5$, one can use (\ref{e:toutmys}) to find $d_a(G)$, $d_b(G)$ and $d_c(G)$ and compute $\textbf{D}_{(a,b,c)}$. Note that if $u\geq h$ then $d_u(G)=\ell$. \\
Finally suppose that $G$ is of type $E_8$ so that $h=30$. The values of $d_u(G)$ are given in \cite{Lawther}. In particular $d_u(G)$ is a decreasing function of $u$. Moreover $\textbf{D}_{(2,3,7)}$,  $\textbf{D}_{(2,4,5)}$ and  $\textbf{D}_{(3,3,4)}$ are all negative, and so $(2,3,7)$, $(2,4,5)$ and $(3,3,4)$ are all nonrigid for $G$. It follows that every hyperbolic triple of integers is nonrigid for $G$. 
$\square$\\

  \section{Reducibility and linear rigidity}\label{s:concepts}
  
  In this section we prove the following result.
  
  \begin{prop}\label{p:lalalaboudou}
  Theorems \ref{t:marionconjsimple} and  \ref{t:cormain} hold provided that  Conjecture \ref{c:marionconj} holds for $G={\rm PSp}_4(K)$ and $(a,b)=(3,3)$.
  \end{prop}
  
  Throughout this section,  $K$ denotes an algebraically closed field of characteristic $p$ (possibly equal to 0, unless otherwise stated). We begin  with the following definition. 
  
  \begin{defn}
Let $H={\rm SL}_n(K)/C$ where $C\leqs Z({\rm SL}_n(K))$, $G$ be a closed  subgroup of $H$ and $(n_1,n_2,n_3)$ be a hyperbolic triple of integers. We say that $(n_1,n_2,n_3)$ is \emph{$G$-reducible for $H$}, \emph{$G$-rigid for $H$} or \emph{$G$-nonrigid for $H$}, respectively, if the maximum value of 
\[\dim g_1^H+\dim g_2^H+\dim g_3^H\]
(where, for $i\in \{1,2,3\}$, $g_i\in G$ has order dividing $n_i$)
is less than, equal to, or greater than $2 \dim H$, respectively.
\end{defn}

Let $G$ be a  closed subgroup of a simple algebraic group $H$ of type $A$ over $K$ and let   $(a,b,c)$ be a hyperbolic triple of integers.  Recall that by an irreducible subgroup of a classical group, we mean a subgroup acting irreducibly on the natural module for the overgroup. We prove below that if  $(a,b,c)$ is not $G$-nonrigid for $H$ then there are up to isomorphism only finitely many (possibly zero) $(a,b,c)$-generated subgroups of $G$ that are irreducible for $H$. The latter result is key  for proving Proposition \ref{p:lalalaboudou}.     We begin with the case where $(a,b,c)$ is $G$-reducible for $H$. 
 
 \begin{lem}\label{l:almostred}
Let $H={\rm SL}_n(K)/C$ where $C\leqs Z({\rm SL}_n(K))$  and  let $G$ be a closed subgroup of $H$. Suppose $(n_1,n_2,n_3)$ is a hyperbolic triple of integers  which is $G$-reducible for $H$. If  $(g_1,g_2,g_3)$ is a triple of elements $g_i$ of $G$ of order dividing $n_i$ such that $g_1g_2g_3=1$ then $\langle g_1,g_2\rangle$ is a reducible subgroup of $H$. 
\end{lem} 

\begin{proof}
Since $(n_1,n_2,n_3)$ is $G$-reducible for $H$, it follows that $\dim g_1^H+\dim g_2^H+\dim g_3^H< 2\dim H$. As $g_1g_2g_3=1$, it now follows from \cite[Proposition 2.1]{Marionconj} that $\langle g_1, g_2\rangle$ is a reducible subgroup of $H$.   
\end{proof}
 
 We continue with the case where $(a,b,c)$ is $G$-rigid for $H$.  The following result is a slight generalisation of \cite[Lemma 3.2]{Marionconj} in which the case where $G=H$ and $a$, $b$ and $c$ are all primes is covered.  
 \begin{lem}\label{l:almostrig} Let $H={\rm SL}_n(K)/C$ where $C\leqs Z({\rm SL}_n(K))$ and let $G\leqs H$ be a simple algebraic group. The following assertions hold. 
\begin{enumerate}[(i)] 
 \item Suppose $H={\rm SL}_n(K)$. If $(a,b,c)$ is a $G$-rigid triple of integers for $H$ then up to
conjugacy in ${\rm GL}_n(K)$, there are only finitely many subgroups of $G$
that are both $(a, b, c)$-generated and  irreducible subgroups of $H$.
\item Suppose $C$ is nontrivial. If $(a,b,c)$ is a $G$-rigid triple of integers for $H$ then up
to isomorphism, there are only finitely many  subgroups of  $G$
that are both $(a,b,c)$-generated and  irreducible subgroups of $H$.
 \end{enumerate}
 \end{lem}

The proof of   Lemma \ref{l:almostrig} is very similar to that of \cite[Lemma 3.2]{Marionconj} and uses the concept of linear rigidity. For further details, we refer the reader to \cite[\S3]{Marionconj}.  For completeness, we recall that a triple  $(g_1, g_2, g_3)$ of elements of ${\rm GL}_n(K)$  with $g_1g_2g_3=1$ is said to be linearly rigid if  for any   triple $(h_1,h_2,h_3)$  of elements of ${\rm GL}_n(K)$
with $h_1h_2h_3=1$ such that   $h_i$ is conjugate to $g_i$  for each  $i$, there exists an element $g$ in ${\rm GL}_n(K)$ with $g_i=gh_ig^{-1}$ for all $i$.
\begin{proof}
  We let $(n_1,n_2,n_3)=(a,b,c)$ and denote by $|h|$ the order of an element $h \in {\rm SL}_n(K)$. \\
  We first consider part (i). Let
$$\mathcal{T}= \{(g_1, g_2, g_3) \in G^3: g_1g_2g_3 =1,  |g_i| \ \textrm{divides} \ n_i,  \langle g_1,  g_2\rangle \ \textrm{is an irreducible subgroup of} \ H \}$$
and for a fixed triple $(C_1,C_2,C_3)$ of conjugacy classes  $C_i$ of $G$ consisting respectively
of elements of orders  dividing $n_1$, $n_2$, $n_3$, let
$$\mathcal{T}_{(C_1,C_2,C_3)}=\{(g_1,g_2, g_3) \in \mathcal{T} : g_i \in  C_i\}.$$ Since $(n_1, n_2, n_3)$ is a $G$-rigid triple of integers for $H$, it follows from \cite[Proposition 2.1]{Marionconj} and
\cite[Lemma 3.1]{Marionconj} that every element of $\mathcal{T}$ is linearly rigid. Hence if $\mathcal{T}_{(C_1,C_2,C_3)}\neq \emptyset$ then
${\rm GL}_n(K)$ is transitive on $\mathcal{T}_{(C_1,C_2,C_3)}$. 
As $G$ is simple,  the number of conjugacy classes $C_i$ of elements
of order  dividing $n_i$ in $G$ is finite. It follows that ${\rm GL}_n(K)$ has finitely many orbits
on $\mathcal{T}$. 
Therefore up to conjugacy in ${\rm GL}_n(K)$, there are only finitely many 
subgroups of $G$ that are $(n_1,n_2,n_3)$-generated and are irreducible subgroups of $H$.\\
Finally let us consider part (ii). Suppose now that $(n_1, n_2, n_3)$ is a $G$-rigid triple of integers for ${\rm SL}_n(K)/C$. Suppose
that $(g_1, g_2, g_3)$ are elements of ${\rm SL}_n(K)$ such that $g_1g_2g_3= \zeta$ for some $\zeta \in C$, $g_iC \in G$ has
order dividing $n_i$, and $\langle g_1,  g_2\rangle$ is an irreducible subgroup of ${\rm SL}_n(K)$. Replacing $g_3$ by $g_3\zeta^{-1}$ we then have $g_1g_2g_3= 1$,
the order $|g_i|$ divides $nn_i$ and $\langle g_1 , g_2\rangle$ is an irreducible subgroup of ${\rm SL}_n(K)$. Furthermore, the $G$-rigidity of $(n_1,n_2, n_3)$ for ${\rm SL}_n(K)/C$ and
\cite[Proposition 2.1]{Marionconj} applied to ${\rm SL}_n(K)/C$  yield
$$\sum \dim g_i^{{\rm SL}_n(K)}
= 2 \dim {\rm SL}_n(K).$$
Set $\hat{\mathcal{T}}$ to be the set of triples $(g_1,g_2,g_3)$ of elements $g_i \in  {\rm SL}_n(K)$ ($1\leq i \leq 3$)  such that $g_1g_2g_3=1$, the order of $g_i$ divides $nn_i$, $g_iC$ is an element of $G$, $ \langle g_1, g_2\rangle $ is an irreducible subgroup of ${\rm SL}_n(K)$ and $$ \sum \dim g_i^{{\rm SL}_n(K)} = 2 \dim {\rm SL}_n(K).$$

By \cite[Lemma 3.1]{Marionconj} every element of $\hat{\mathcal{T}}$ is linearly rigid, and a similar argument to that
given in (i) above shows that ${\rm GL}_n(K)$ has finitely many orbits on $\hat{\mathcal{T}}$. Therefore up
to isomorphism there are only finitely many  subgroups of $G$ that
are $(n_1, n_2, n_3)$-generated and are irreducible subgroups of $H$. 
\end{proof}

As an immediate consequence of Lemma \ref{l:almostrig} applied to the special case where $G=H$, we obtain the following result.

  \begin{prop}\label{p:marionconjtypea}
  Let $K$ be an algebraically closed field  of prime characteristic. If $G={\rm SL}_n(K)/C$ where $C\leqs Z({\rm SL}_n(K))$ then Conjecture \ref{c:marionconj}  holds. 
  \end{prop}

In the next two results we make further progress on Conjecture \ref{c:marionconj} in the special case where $G$ is of type $C_2$ and $(a,b)=(2,3)$ or $G$ is of type $G_2$. In doing so we use Lemmas \ref{l:almostred} and \ref{l:almostrig}.  

 \begin{prop}\label{p:marionconjtypec2one} 
 If $G={\rm PSp}_4(K)$ and $(a,b)=(2,3)$ then Conjecture \ref{c:marionconj} holds. 
  \end{prop}
  
  \begin{proof}
 In \cite[Proposition 6.2]{Liebeck96} it is shown that ${\rm PSp}_4(p^r)$ is never $(2,3)$-generated when $p\in \{2,3\}$. We can therefore assume that $p\not \in \{2,3\}$. By Theorem \ref{t:classification} the triple $(2,3,c)$ which is rigid for ${\rm PSp}_4(K)$ remains rigid for ${\rm PSL}_4(K)$. The result now follows from Lemma  \ref{l:almostrig}. 
  \end{proof}

\begin{prop}\label{p:marionconjtypeg2one}
  Suppose $G=G_2(K)$, $p \in \{2,5\}$ and  $(a,b,c)\in \{(2,4,5), (2,5,5)\}$.
  \begin{enumerate}[(i)]
    \item If $p=5$ or $(a,b,c)=(2,4,5)$ then  there are no positive integers $r$ such that $G_2(p^r)$ is an $(a,b,c)$-group.
  \item If $p=2$ and $(a,b,c)=(2,5,5)$ then  there are only finitely many positive integers $r$ such that $G_2(p^r)$ is an $(a,b,c)$-group.
  \end{enumerate}  
  \end{prop}

\begin{proof}
Note that in \cite{Mariong2} it was shown that $G_2(5^r)$ is
never a $(2, 5,5)$-group. For completeness we include a slightly different proof of
this latter fact. \\
Write $G=G_2(K)$ where $K$ is an algebraically closed field of  characteristic $p\in \{2,5\}$.\\
 Suppose first that $p=2$.   Then $G$ is a subgroup of ${\rm SL}_6(K)$. 
By \cite{Lawther95}, an involution in $G_2(K)$ is conjugate in ${\rm SL}_6(K)$ to an element $g_{2,1}$ or $g_{2,2}$ having Jordan form  $J_1^2\oplus J_2^2$ or $J_2^3$ respectively, and an element of $G_2(K)$ of order 4 is conjugate in ${\rm SL}_6(K)$ to an element $g_{4,1}$ having Jordan form $g_{4,1}=J_3^2$.\\
 By \cite[p. 367]{MoodyPatera} there are four conjugacy classes in $G_2(K)$  of elements of order 5. Moreover all these classes are real (see \cite[Theorem 1.5]{Tiep}). However, by \cite[Table 10]{MoodyPatera} none of these classes are rational. The representatives of these four classes are therefore of the form: $\alpha$, $\alpha^2$, $\beta$ and $\beta^2$. 
 Finally, following \cite[p.34] {Cooperstein},  we deduce that an element of order $5$ in $G_2(K)$ is conjugate in ${\rm SL}_6(K)$ to an element of the form
$$ g_{5,1}={\rm diag}(\omega,\omega^{-1},\omega^2,\omega^{-2}, \omega, \omega^{-1}) \quad \textrm{or} \quad g_{5,2}={\rm diag}(\omega,\omega,1,1,\omega^{-1},\omega^{-1})$$ (or the square of such an element) where $\omega \in K$ is a primitive fifth root of unity. 
Following \cite[Theorem 3.1]{LiebeckSeitz} we have
$$ \dim g_{2,1}^{{\rm SL}_6(K)} = 16, \quad   \dim g_{2,2}^{{\rm SL}_6(K)} = 18 \quad \textrm{and} \quad \dim g_{4,1}^{{\rm SL}_6(K)} = 24.$$
Also an easy calculation yields: 
 $$\dim g_{5,1}^{{\rm SL}_6(K)} = 26 \quad \textrm{and} \quad   \dim g_{5,2}^{{\rm SL}_6(K)} = 24.$$
 In particular   $(2,4,5)$  is $G_2(K)$-reducible for ${\rm SL}_6(K)$. By Lemma \ref{l:almostred} a $(2,4,5)$-generated subgroup of $G$ is  a reducible subgroup of ${\rm SL}_6(K)$. It follows that there are no positive integers $r$ such that $G_2(2^r)$ is a $(2,4,5)$-group.\\
 Also if $(a,b,c)=(2,5,5)$ then $(a,b,c)$ is $G_2(K)$-rigid for ${\rm SL}_6(K)$. It now follows from Lemma \ref{l:almostrig} that there are only finitely many positive integers $r$ such that $G_2(2^r)$ is a $(2,5,5)$-group. \\
 
 Suppose finally that $p=5$.  Then $G$ is a subgroup of ${\rm SL}_7(K)$. 
By \cite{Lawther95}, an element of order 5 in $G_2(K)$ is conjugate in ${\rm SL}_7(K)$ to an element $g_{5,1}$, $g_{5,2}$ or $g_{5,3}$ having Jordan form  $J_1^3\oplus J_2^2$, $J_1^1\oplus J_2^3$,  or $J_1^1\oplus J_3^2$ respectively. 
Following \cite[Theorem 3.1]{LiebeckSeitz} we have
 $$ \dim g_{5,1}^{{\rm SL}_7(K)} = 20, \ \dim g_{5,2}^{{\rm SL}_7(K)} = 24 \quad \textrm{and} \quad  \dim g_{5,3}^{{\rm SL}_7(K)} = 32.$$
Also an element $g_{2,1}$ of ${\rm SL}_7(K)$ of order 2 whose centralizer in ${\rm SL}_7(K)$ has minimal dimension satisfies $ \dim g_{2,1}^{{\rm SL}_7(K)}=24$ and an element $g_{4,1}$ of  ${\rm SL}_7(K)$ of order 4 whose centralizer in ${\rm SL}_7(K)$ has minimal dimension satisfies  $\dim g_{4,1}^{{\rm SL}_7(K)}=36$. 
Indeed one can take $g_{2,1}={\rm diag}(1,1,1,-1,-1,-1,-1)$ and $g_{4,1}={\rm diag}(1,-1,-1,\omega,\omega,\omega^{-1},\omega^{-1})$ where $\omega \in K$ is a primitive fourth root of unity.
Hence if $(a,b,c)\in\{(2,4,5),(2,5,5)\}$  then $(a,b,c)$ is $G_2(K)$-reducible for ${\rm SL}_7(K)$. By Lemma \ref{l:almostred} an $(a,b,c)$-generated subgroup of $G$ is  a reducible subgroup of ${\rm SL}_7(K)$. It follows that there are no positive integers $r$ such that $G_2(5^r)$ is an $(a,b,c)$-group.
\end{proof}


\noindent\textit{Proof of Proposition \ref{p:lalalaboudou}.} The result now follows immediately from Corollaries \ref{c:s2c1} and \ref{c:s2c2} together with Propositions \ref{p:marionconjtypea}-\ref{p:marionconjtypeg2one}.


\section{$(3,3,c)$-generation for ${\rm Sp}_4(p^r)$ and ${\rm PSp}_4(p^r)$}\label{s:sp433}

In this section, given a positive integer $c$ and a prime number $p$ we prove that there are only finitely many positive integers $r$ such that the finite groups ${\rm Sp}_4(p^r)$ and ${\rm PSp}_4(p^r)$ are  $(3,3,c)$-groups. This establishes Conjecture \ref{c:marionconj} in the special case where $G={\rm Sp}_4(K)$ or ${\rm PSp}_4(K)$ with $K$ an algebraically closed field of prime characteristic and $(a,b)=(3,3)$.  The case studied in this section requires a different approach to that of \S\ref{s:concepts}. Indeed, since $(3,3,c)$ with $c\geq5$ (respectively, with $c\geq 4$) is ${\rm Sp}_4(K)$-nonrigid (respectively, ${\rm PSp}_4(K)$-nonrigid) for ${\rm SL}_4(K)$ (respectively, ${\rm PSL}_4(K)$), one can no longer make use of Lemma \ref{l:almostred} or Lemma \ref{l:almostrig}.

\begin{thm}\label{t:33sp}
Let $p$ be a prime number and $c$ be a positive integer. The following assertions hold.
\begin{enumerate}[(i)]
\item There are only finitely many positive integers $r$ such that ${\rm Sp}_4(p^r)$ is a $(3,3,c)$-group.  
\item There are only finitely many positive integers $r$  such that ${\rm PSp}_4(p^r)$ is a $(3,3,c)$-group.
\item In particular, if $G={\rm Sp}_4(K)$ or   ${\rm PSp}_4(K)$, and $(a,b)=(3,3)$ then Conjecture \ref{c:marionconj} holds. Hence Theorems \ref{t:marionconjsimple} and \ref{t:cormain} also hold.
\end{enumerate}
\end{thm}


Unless otherwise stated,  in this section $\mathbb{F}$ denotes an arbitrary field and we denote  by $\mathbb{F}_q$ a finite field of order  a prime power $q=p^r$. Let $n$ be a positive integer. Recall that for a subgroup $H$ of ${\rm GL}_n(\mathbb{F})$,  the character field ${\rm cf}(H)$ of $H$ is   the subfield of $\mathbb{F}$ generated by $\{{\rm tr}(h): h \in H\}$. \\


The remainder of the paper is dedicated to the proof of Theorem \ref{t:33sp}. Since our methods differ somewhat from the preceding material, let us give a brief outline. If $g_1$, $g_2 \in {\rm Sp}_{4}(q)$, the character field ${\rm cf}(\left<g_1,g_2\right>)$ is the subfield of $\mathbb{F}_{q}$ generated by traces of all words in $g_1$ and $g_2$. We view the traces of words as an elements of a polynomial ring, and use the further assumption $g_1^3 = g_2^3 = 1$ to derive relations between these polynomials which hold in ${\rm cf}(\left<g_1,g_2\right>)$. By assuming also that $\left<g_1,g_2\right>$ is absolutely irreducible, we ultimately show in Proposition \ref{p:chf} below that ${\rm cf}(\left<g_1,g_2\right>)$ is generated by the values of just \emph{three} such polynomials and find also a relation between these three polynomials. Hence ${\rm cf}(\left<g_1,g_2\right>)$ is a quotient of a finitely-generated polynomial ring. By using the final assumption that $g_1g_2$ has order dividing a given positive integer $c$, we derive two further relations for the generators (see Lemma \ref{l:chi} below). This enables us to   show that ${\rm cf}(\left<g_1,g_2\right>)$ is an image of a particular finite quotient of the polynomial ring, which does not depend on $q$. In particular the size of ${\rm cf}(\left<g_1,g_2\right>)$ is bounded. Since we also show below that ${\rm cf}({\rm Sp}_{4}(q)) = \mathbb{F}_{q}$, the conclusion of Theorem \ref{t:33sp} then follows.

\begin{lem}\label{l:chf}
If $G_0={\rm Sp}_4(q)={\rm Sp}_4(p^r)$ then ${\rm cf}(G_0)=\mathbb{F}_q$. 
\end{lem}

\begin{proof}
Let $V$ be the natural module for ${\rm Sp}_4(\mathbb{F}_q)$ and let $\mathcal{B}=(e_1,f_1,e_2,f_2)$ be a standard basis of $V$. That is, $(e_i,f_i)=1$, $(e_i,e_i)=(f_i,f_i)=0$ for $1\leq i\leq 2$, $(e_i,f_j)=(e_i,e_j)=(f_i,f_j)=0$ for  $i\neq j$ such that $1\leq i, j\leq 2$. Here $(,): V\times V \rightarrow \mathbb{F}_q$ denotes the  symplectic form associated to $V$.  \\
Let $\omega$ be a generator of the cyclic group $\mathbb{F}_q^{\ast}$ and let $g$ be the element of ${\rm GL}_4(q)$ defined by:\\

$\begin{array}{ll}
g(e_1)=\omega e_1-f_2\\
g(f_1)=-\omega e_1+e_2+f_2\\
g(e_2)=e_1\\
g(f_2)=e_1+f_1-\omega e_2.
\end{array}$

It is easy to check that $g$ is in fact an element of $G_0={\rm Sp}_4(q)$ and that ${\rm tr}(g)=\omega$.  It follows that ${\rm cf}(G_0)=\mathbb{F}_q$ as claimed. 
\end{proof}

Let ${\rm M}_n(R)$ denote the algebra of $n\times n$ matrices over a commutative ring $R$.   Given a ring homomorphism $\alpha: R\rightarrow S$, note that  $\alpha$ induces a ring homomorphism ${\rm M}_n(R)\rightarrow {\rm M}_n(S)$ where $\alpha(A)_{ij}=\alpha(A_{ij})$ for all $A\in {\rm M}_n(R)$. We recall the following elementary result, see for example \cite[XIV, \S3]{Lang}.

\begin{lem}\label{l:mapchp}
Let $\alpha: R\rightarrow S$ be  a ring homomorphism and consider the induced ring homomorphism $\alpha: {\rm M}_n(R)\rightarrow {\rm M}_n(S)$. Let $A\in {\rm M}_n(R)$, $B=\alpha(A)$, and  $p_A$ and $p_B$ be respectively the characteristic polynomials of $A$ and $B$. Then $p_B(T)=\alpha(p_A)(T)$ where $\alpha(p_A)$ is the polynomial obtained from $p_A$ by applying $\alpha$ to the coefficients of $p_A$. 
\end{lem}

Given $M\in {\rm M}_4(R)$, we let $p_M(T)$ be its characteristic polynomial. Write 
\begin{equation}\label{e:charpoly}
p_M(T)=T^4-\chi_3(M)T^3+\chi_2(M)T^2-\chi_1(M)T+\chi_0(M)
\end{equation}
where $\chi_i(M)\in R$ for $0\leq i \leq 3$. Note that $\chi_3(M)={\rm tr}(M)$ and
 ${\rm tr}(M^2)={\rm tr}(M)^2-2\chi_2(M)$.

\begin{lem}
Let $M_1,M_2,M_3\in {\rm GL}_4(\mathbb{F})$ and define $M_{-l}:=M_l^{-1}$ for $l\in \{1,2,3\}$. For a positive integer $k$ and a $k$-tuple $(i_1,\dots, i_k)$ with $i_j\in\{\pm1,\pm2,\pm 3\}$ for $1\leq j\leq k$, set 
$$t_{i_1\dots i_k}:={\rm tr}(M_{i_1}\dots M_{i_k})$$ and 
$$c_{i_1\dots i_k}:=\chi_2(M_{i_1}\dots M_{i_k}).$$ Then the following assertions hold.
\begin{enumerate}[(i)]
\item
\begin{equation}\label{e:procesi1}
\begin{array}{lll}
0 & = &  t_{12123}+t_{21213}+t_{11223}+t_{12213}+t_{21123}+t_{22113}\\
& & -t_1(t_{1223}+t_{2123}+t_{2213})-t_2(t_{1123}+t_{1213}+t_{2113})+(t_1t_2-t_{12})(t_{123}+t_{213})\\
& & +c_2t_{113}+c_1t_{223}+(t_{12}t_2-t_{122}-t_1c_2)t_{13}+(t_{12}t_1-t_{112}-t_2c_1)t_{23}\\
& & -(t_{1122}-t_{112}t_2-t_{122}t_1+t_{12}t_1t_2-c_1c_2-c_{12})t_3.
\end{array}
\end{equation}
\item \begin{equation}\label{e:procesi2}
\begin{array}{lll}
0 & = & t_{-12213}+t_{2-1213}+t_{-12123}+t_{12-123}+t_{212-13}+t_{122-13}\\
&  & -(t_{12-13}+t_{-1213})t_2-(t_{2123}+t_{1223}+t_{2213})t_{-1}-(t_{-1223}+t_{22-13}+t_{2-123})t_1\\
&&-(t_{123}+t_{213})(t_{-12}-t_2t_{-1})-(t_{-123}+t_{2-13})(t_{12}-t_2t_1)+t_{223}(t_{-1}t_1+2)\\
& & -(t_{-122}-t_{-12}t_2+c_2t_{-1})t_{13}-(t_{122}-t_{12}t_2)t_{-1}-(t_{-122}-t_{-12}t_2)t_1-c_2(t_{-1}t_1+2)t_3. 
\end{array}
\end{equation}
\item\begin{equation}\label{e:procesi3}
\begin{array}{lll}
0 & = & t_{21-2-13}+t_{-212-13}+t_{12-1-23}+t_{-1-2123}+t_{-2-1213}+t_{-121-23}+t_{1-2-123}+t_{2-1-213}\\
&&-(t_{2-1-23}+t_{-2-123})t_1-(t_{-1-213}+t_{1-2-13})t_2-(t_{21-23}+t_{-2123})t_{-1}-(t_{-1213}+t_{12-13})t_{-2}\\
& & -(t_{-1-23}+t_{-2-12})(t_{12}-t_1t_2)-(t_{213}+t_{123})(t_{-2-1}-t_{-1}t_{-2})\\
& & -(t_{-213}+t_{1-23})(t_{-12}-t_2t_{-1})-(t_{-123}+t_{2-13})(t_{-21}-t_1t_{-2})\\
& & +(t_{-2-1}t_1+t_{-21}t_{-1}-(t_1t_{-1}+2)t_{-2})t_{23}+(t_{-12}t_1+t_{12}t_{-1}-(t_1t_{-1}+2)t_2)t_{-23}\\
& & +(t_{-2-1}t_2+t_{-12}t_{-2}-(t_2t_{-2}+2)t_{-1})t_{13}+(t_{-21}t_2+t_{12}t_{-2}-(t_2t_{-2}+2)t_1)t_{-13}\\
&& - (t_{-212-1}+t_{-2-121}-(t_{-2-1}-t_{-1}t_{-2})(t_{12}-t_1t_2)-(t_{-21}-t_{-2}t_1)(t_{-12}-t_2t_{-1}))t_3\\
&&+((t_1t_{-1}-2)(t_2t_{-2}-2)-4)t_3.
\end{array}
\end{equation}
\end{enumerate}
\end{lem}
    
\begin{proof}
Assume first that     ${\rm char}(\mathbb{F})=0$. We use a special case of a result of Procesi \cite[Theorem 4.3]{Procesi}. 
Let $Z_1,\dots, Z_5\in {\rm M}_{4}(\mathbb{F})$. Let $\sigma \in {\rm Sym}_5$. Write  $$\sigma=(\sigma_1,\dots, \sigma_{r_1})\dots (\sigma_{r_i},\dots,\sigma_5)$$
 a cycle decomposition of $\sigma$ into disjoint cycles  including trivial cycles.  Set $${\rm tr}_{\sigma}(Z_1,\dots,Z_5)={\rm tr}(Z_{\sigma_1}\dots Z_{\sigma_{r_1}})\dots {\rm tr}(Z_{\sigma_{r_i}}\dots Z_{\sigma_{5}}).$$ Then $$\sum_{\sigma\in {\rm Sym}_5} {\rm sgn}(\sigma){\rm tr}_{\sigma}(Z_1,\dots,Z_5)=0.$$
 
 Let $Z_1=M_1$, $Z_2=M_2$ and $Z_3=M_3$. The result follows by setting $Z_4=M_1$ and $Z_5=M_2$ for Equation (\ref{e:procesi1}), $Z_4=M_1^{-1}$ and $Z_5=M_2$ for Equation (\ref{e:procesi2}), and $Z_4=M_1^{-1}$ and $Z_5=M_2^{-1}$ for Equation (\ref{e:procesi3}), using in the three cases that ${\rm tr}(M^2)={\rm tr}(M)^2-2\chi_2(M)$ for all $M\in {\rm GL}_4(\mathbb{F})$.\\
 
 Suppose now that $\mathbb{F}$ is arbitrary.  Let $X_{ij}^l$, $d_l$ be indeterminates where $1\leq i, j \leq 4$ and $1\leq l\leq 3$. Then, for $1\leq l \leq 3$, $p_l={\rm det}(X_{ij}^l)d_l-1$ is an irreducible polynomial of $\mathbb{Z}[X_{ij}^l, d_l: 1\leq i, j \leq 4, 1\leq l \leq 3]$, and the ideal $I=(p_1,p_2,p_3)$ of  $$\mathbb{Z}[X_{ij}^l, d_l: 1\leq i, j \leq 4, 1\leq l \leq 3]$$ generated by $\{p_1,p_2,p_3\}$ is prime. Hence $$R=\mathbb{Z}[X_{ij}^l, d_l: 1\leq i, j \leq 4, 1\leq l \leq 3]/I$$ is an integral domain. Let $F$ be the field of fractions of $R$, a field of   characteristic 0. 
 Consider the elements $(A_{ij}^1)$, $(A_{ij}^2)$ and $(A_{ij}^3)$ of ${\rm M}_4(R)$ defined by  $A_{ij}^l=X_{ij}^l+I$ for  $1\leq i,j\leq 4$ and $1\leq l\leq 3$.  Note that 
 $(A_{ij}^1)$, $(A_{ij}^2)$ and $(A_{ij}^3)$ belong to ${\rm GL}_4(F)$.   Since ${\rm char}(F)=0$,  the result holds for the elements $(A_{ij}^1)$, $(A_{ij}^2)$ and $(A_{ij}^3)$. \\
 Now we can choose a ring homomorphism $\alpha: R\rightarrow \mathbb{F}$ such that $\alpha(X_{ij}^l+I)=(M_{l})_{ij}$  for  $1\leq i,j\leq 4$ and $1\leq l\leq 3$. Consider the induced ring homomorphism $\alpha: {\rm M}_4(R)\rightarrow {\rm M}_4(\mathbb{F})$ such that $\alpha(A)_{ij}=\alpha(A_{ij})$ for all $A\in {\rm M}_4(R)$ and $1\leq i, j\leq 4 $.
  In particular, we have 
 $M_l=\alpha(A_{ij}^l)$ for $1\leq l\leq 3$.  Lemma \ref{l:mapchp} together  with the fact that the result holds for   $(A_{ij}^1)$, $(A_{ij}^2)$ and $(A_{ij}^3)$ in ${\rm GL}_4(F)$ now yield the result in general.  \end{proof}

Let $F_2$ be the free group on two generators $x_1$ and $x_2$. Given a word $w$ of $F_{2}$ and a group $G$,  the word map $w: G\times G\rightarrow G$ is defined such that  for every $(g_1,g_2)\in G\times G$, $w(g_1,g_2)$ is obtained from $w(x_1,x_2)$ by substituting the group elements $g_1$, $g_2$ for $x_1$, $x_2$.  For an element $w \in F_2$ we denote by $|w|$ its length with respect to the generators $x_1$ and $x_2$ of $F_2$, whereas given an element  $ g \in G$ we let $|g|$ denote the order of $g$.\\

\begin{prop}\label{p:chfgeneration}
Let $g_1, g_2$ be elements of  ${\rm GL}_4(\mathbb{F})$ of order $3$. The character field of $ \langle g_1,g_2\rangle$ is generated by 
$$X=\{{\rm tr}(w(g_1,g_2)): w \in F_2, |w|\leq 4\}\cup \{\chi_2(w(g_1,g_2)): w \in F_2, |w|\leq 2\}.$$
\end{prop}

\begin{proof}
It is enough to  show that for every $w\in F_2$, ${\rm tr}(w(g_1,g_2))$ is a polynomial in $X$ with coefficients in $\mathbb{Z}$.  We proceed by induction on $|w|$. The claim is obvious for $|w|\leq 4$. Assume $|w|>4$. Since $|g_1|=|g_2|=3$ and since the trace is invariant under cyclic permutations of a product, we may assume 
\begin{equation}\label{e:wordeq2016}
w=x_1^{a_1}x_2^{b_1}\cdots x_1^{a_r}x_2^{b_r}
\end{equation} with $a_i, b_i \in \{\pm 1\}$. Then
$$w(g_1,g_2)=g_1^{a_1}g_2^{b_1}\cdots g_1^{a_r}g_2^{b_r}.$$
Note that a word $w$ of length $|w|>4$ as in (\ref{e:wordeq2016}) must contain, up to cyclic permutation, a subword of the form $y_1y_2y_1y_2$ or $y_1y_2y_1^{-1}y_2$ or $y_1y_2y_1^{-1}y_2^{-1}$ where $\{y_1, y_1^{-1},y_2,y_2^{-1}\}=\{x_1,x_1^{-1},x_2,x_2^{-1}\}$.\\
Assume first that $w=x_1x_2x_1x_2x_1^{a_3}\cdots x_2^{b_r}$. Set $M_1=g_1$, $M_2=g_2$ and $M_3=g_1^{a_3}\cdots g_2^{b_r}$. Then 
$$t_{21213}={\rm tr}(M_2M_1M_2M_1M_3)={\rm tr}(\tilde{w}(g_1,g_2))$$ where $\tilde{w}\in F_2$ is such that $\tilde{w}(g_1,g_2)=g_2g_1g_2g_1^{a_3+1}g_2^{b_3}\cdots g_2^{b_r}$ and  $|\tilde{w}|\leq |w|-1$.
By Equation (\ref{e:procesi1}) and induction ${\rm tr}(w(g_1,g_2))=t_{12123}$ can be written as a polynomial in $X$ with coefficients in $\mathbb{Z}$. 
Similarly, as the trace is invariant under cyclic permutations, if $w$ or one of its cyclic permutations contains a subword of the form $y_1y_2y_1y_2$, where $\{y_1, y_1^{-1},y_2,y_2^{-1}\}=\{x_1,x_1^{-1},x_2,x_2^{-1}\}$ then ${\rm tr}(w(g_1,g_2))$ can be written as a polynomial in $X$ with coefficients in $\mathbb{Z}$. In the remainder, we may therefore assume that neither $w$ nor any of its cyclic permutations contain a subword of the form $y_1y_2y_1y_2$ where $\{y_1, y_1^{-1},y_2,y_2^{-1}\}=\{x_1,x_1^{-1},x_2,x_2^{-1}\}$.\\
Assume next that $w=x_1x_2x_1^{-1}x_2x_1^{a_3}\cdots x_2^{b_r}$. By the above, we may assume that $a_3=1$. Set $M_1=g_1$, $M_2=g_2$ and $M_3=g_1g_2^{b_3}\cdots g_2^{b_r}$. By Equation (\ref{e:procesi2})   $t_{12-123}+t_{-12123}$ can be written as a polynomial in $X$ with coefficients in $\mathbb{Z}$. Now $M_1^{-1}M_2M_1M_2M_3$ contains the subword $g_2g_1g_2g_1$ so $t_{-12123}={\rm tr}(M_1^{-1}M_2M_1M_2M_3)$ can be written as a polynomial in $X$ with coefficients in $\mathbb{Z}$, hence so can be ${\rm tr}(w(g_1,g_2))=t_{12-123}$.\\
By the two cases considered above, in the remainder, we may assume that neither $w$ nor any of its cyclic permutations contain a subword of the form $y_1y_2y_1y_2$ or $y_1y_2y_1^{-1}y_2$ where $\{y_1, y_1^{-1},y_2,y_2^{-1}\}=\{x_1,x_1^{-1},x_2,x_2^{-1}\}$.\\
Assume now that $w=x_1x_2x_1^{-1}x_2^{-1}x_1^{a_3}\cdots x_2^{b_r}$.  By the above,  we may assume $a_3=b_3=1$.  Set $M_1=g_1$, $M_2=g_2$ and $M_3=g_1g_2\cdots g_2^{b_r}$. Using Equation (\ref{e:procesi3}), a similar argument to the two above yields that ${\rm tr}(w(g_1,g_2))=t_{12-1-23}$ can be written as a polynomial in $X$ with coefficients in $\mathbb{Z}$.   Similarly, if $w$,  or a cyclic permutation of $w$, contains a subword of the form $y_1y_2y_1^{-1}y_2^{-1}$ where $\{y_1, y_1^{-1},y_2,y_2^{-1}\}=\{x_1,x_1^{-1},x_2,x_2^{-1}\}$ then ${\rm tr}(w(g_1,g_2))$ can be written as a polynomial in $X$ with coefficients in $\mathbb{Z}$.
This concludes the proof.
\end{proof}

In the statement below, by an (absolutely) irreducible subgroup of a classical group $G$, we mean a subgroup acting (absolutely) irreducibly on the natural module for $G$.

\begin{lem}\label{l:wedd}
Let $g_1, g_2\in {\rm GL}_n(\mathbb{F})$  where $n\geq 4$. Suppose $g_1$ and $g_2$ both have minimal polynomials of degree $2$. Then $\langle g_1, g_2\rangle$ is not absolutely irreducible.  
\end{lem}

\begin{proof}
 Suppose $\langle g_1,g_2 \rangle$ is absolutely irreducible. Then, as a consequence of Wedderburn's Theorem,    ${\rm M}_n(\mathbb{F})$ has a $\mathbb{F}$-basis of elements of the form $w(g_1,g_2)$ where $w\in F_2$.   Since, for $i \in \{1,2\}$, $g_i$ has minimal polynomial of degree $2$, and $g_1g_2$ and $g_2g_1$ have characteristic polynomials of degree $n$, we deduce that $$S=\{1,g_1, g_2, (g_1g_2)^m, (g_1g_2)^mg_1, (g_2g_1)^m, (g_2g_1)^mg_2: 1\leq m \leq n-1 \}$$ spans  ${\rm M}_n(\mathbb{F})$ over $\mathbb{F}$. But $$|S|=4(n-1)+3=4n-1<n^2 =\dim {\rm M}_n(\mathbb{F}),$$ a contradiction.    
  \end{proof}

\begin{prop}\label{p:chf}
Let $g_1,g_2 \in {\rm Sp}_4(\mathbb{F})$ be elements of order $3$ such that $\langle  g_1,g_2\rangle$ is absolutely irreducible. Set $t_{12}={\rm tr}(g_1g_2)$, $t_{1-2}={\rm tr}(g_1g_2^{-1})$ and $c_{12}=\chi_2(g_1g_2)$. The character field of $\langle  g_1, g_2\rangle$ is generated by $t_{12}$, $c_{12}$ and $t_{1-2}$. Moreover 
\begin{equation}\label{e:gencharfield}
( t_{1-2}+t_{12}+1)(t_{1-2}^2+(2t_{12}-10)t_{1-2}+t_{12}^2-9c_{12}+8t_{12}+7)=0.
\end{equation}
\end{prop}

\begin{proof}
Since we can argue up to conjugation in ${\rm GL}_4(\overline{\mathbb{F}})$  where $\overline{\mathbb{F}}$ denotes an algebraic closure of $\mathbb{F}$, we may assume without loss of generality that $\mathbb{F}$ is algebraically closed. 
For a positive integer $k$ and a $k$-tuple $(i_1,\dots, i_k)$ with $i_j\in\{\pm1,\pm2\}$ for $1\leq j\leq k$, set 
$$t_{i_1\dots i_k}:={\rm tr}(g_{i_1}\dots g_{i_k})$$ and 
$$c_{i_1\dots i_k}:=\chi_2(g_{i_1}\dots g_{i_k}),$$ where $$g_{-l}:=g_l^{-1}$$ for $l\in \{1,2\}$. If ${\rm char}(\mathbb{F})\neq 3$ we let $\omega \in \mathbb{F}$ be a primitive third root of unity.\\
Let $g$ be an element of  ${\rm Sp}_4(\mathbb{F})$ of  order 3. Note that 
$$p_g(T)=\left\{\begin{array}{ll} 
(T-\omega)^2(T-\omega^{-1})^2 & \textrm{if} \ {\rm char}(\mathbb{F})\neq 3 \ \textrm{and} \ 1 \ \textrm{is not an eigenvalue of } \ g \\
(T-1)^2(T-\omega)(T-\omega^{-1}) & \textrm{if} \ {\rm char}(\mathbb{F})\neq 3 \ \textrm{and} \ 1 \ \textrm{is  an eigenvalue of } \ g \\
(T-1)^4  & \textrm{if} \ {\rm char}(\mathbb{F})= 3. 
 \end{array}\right.$$  
 In particular, by (\ref{e:charpoly}),   ${\rm tr}(g)\in\{-2,1\}$ and $\chi_2(g)\in\{0,3\}$.\\
More generally   if $g,h\in {\rm GL}_4(\mathbb{F})$ then  $gh$ and $hg$ have the same characteristic polynomial. Moreover since ${\rm tr}(g^{-1})={\rm tr}(g)$ and $\chi_2(g^{-1})=\chi_2(g)$  for every $g\in {\rm Sp}_4(\mathbb{F})$, we now deduce from Proposition \ref{p:chfgeneration} that the character field of $ \langle g_1,g_2\rangle$ is generated by $$\{c_{12}, c_{1-2}, t_{12}, t_{1-2},t_{1212},t_{121-2},t_{12-12},t_{12-1-2},t_{1-21-2}\}.$$
Therefore it is sufficient to show that $c_{1-2}$, $t_{1212}$,  $t_{121-2}$, $t_{12-12}$, $t_{12-1-2}$ and  $t_{1-21-2}$ can be written as polynomials in $t_{12}$, $c_{12}$ and $t_{1-2}$ with coefficients in $\mathbb{Z}$, and that 
the relation in (\ref{e:gencharfield}) is satisfied.\\
We first note that $$t_{1212}={\rm tr}((g_1g_2)^2)= {\rm tr}(g_1g_2)^2-2\chi_2(g_1g_2)=t_{12}^2-2c_{12}$$ and so $t_{1212}$ is a polynomial in $t_{12}$ and $c_{12}$ with integer coefficients. \\
Note also that  for an element $g$ of ${\rm Sp}_4(\mathbb{F})$ of order $3$ we have 
$$p_g(T)=\left\{\begin{array}{ll} 
(T^2+T+1)^2& \textrm{if} \ {\rm char}(\mathbb{F})= 3 \ \textrm{or} \ 1 \ \textrm{is not an eigenvalue of } \ g \\
(T-1)^2(T^2+T+1) & \textrm{if} \ {\rm char}(\mathbb{F})\neq 3 \ \textrm{and} \ 1 \ \textrm{is  an eigenvalue of } \ g \\
 \end{array}\right.$$  
and the  minimal polynomial of $g$ is  $(T-1)(T^2+T+1)$ or $T^2+T+1$   according respectively as whether or not  ${\rm char}(\mathbb{F})\neq 3$ and $1$ is an eigenvalue of $g$.\\
  Let $V$ be the natural module for ${\rm Sp}_4(\mathbb{F})$.  Since $\langle g_1,g_2\rangle$ is irreducible, Lemma \ref{l:wedd} yields that at least one of $g_1$ or $g_2$ has minimal polynomial of degree 3. Without loss of generality, we may assume that $g_1$ has minimal polynomial of degree 3. \\
  Let $W\leq V$ be a two-dimensional eigenspace for $g_2$. \\

We now consider two cases:
\begin{enumerate}[(a)]
\item There is a nonzero element $v$ of $W$ such that $g_1^2v \in \langle v, g_1v\rangle$.
\item The triple $(v,g_1v, g_1^2v)$ is linearly independent for every $0\neq v \in W$.
\end{enumerate}

\textbf{Case (a).} Assume first that there is a nonzero element $v$ of $W$ such that $g_1^2v \in \langle v, g_1v\rangle$. Let $U=\langle v,g_1v\rangle$.  Note that $\dim U=2$.  Indeed, otherwise $v$ is an eigenvector of $g_1$ and $g_2$, contradicting the irreducibility of $\langle g_1, g_2\rangle$. Note also that $U$ is a $\langle g_1 \rangle$-invariant subspace. Let $g_U$ and $g_{V/U}$ be the linear maps induced by $g_1$ on $U$ and $V/U$, respectively.  Then $g_U$ has characteristic polynomial $T^2+T+1$ and so $g_{V/U}$ must have characteristic polynomial $(T-1)^2$.  Thus $g_{V/U}$ has minimal polynomial $T-1$ or $(T-1)^2$, and the latter case can only occur if ${\rm char}(\mathbb{F})=3$.\\
Note that $v, g_1v$ and $g_2g_1v$ are linearly independent. Indeed otherwise $g_2g_1v \in \langle v, g_1v \rangle$ and so  $U$ is $\langle g_1, g_2 \rangle$-invariant, contradicting the irreducibility of $\langle g_1,g_2\rangle$.\\
Let $w\in W$ be such that $W=\langle v,w \rangle$. Note that $w\not \in U$. Indeed, otherwise $U$ is again $\langle g_1, g_2 \rangle$-invariant, a contradiction.\\

\textbf{Subcase (a.1).} Suppose that $g_{V/U}$ has minimal polynomial $T-1$.   Since $g_{V/U}$ has minimal polynomial $T-1$ and $w\not \in U$, we deduce that $g_1w=w+u$ for some $u\in U$. We claim that $B=(v, g_1v,g_2g_1v,w )$ is a basis of $V$. Suppose not. Then $w\in \langle v, g_1v, g_2g_1v \rangle$ and so $\langle  v, g_1 v, g_2g_1v\rangle=\langle  v, g_1v, w\rangle$  is $\langle g_1,g_2 \rangle$-invariant,  contradicting the irreducibility of $\langle g_1,g_2\rangle$. With respect to the basis $B$ of $V$, the elements $g_1$ and $g_2$ have the following form
$$ g_1=\begin{pmatrix} 0 & -1 & a_{13} & a_{14}\\
1& -1 & a_{23} &a_{24}\\
0 & 0 & 1 & 0\\
0 & 0 & 0 & 1
 \end{pmatrix}, \quad 
 g_2=\begin{pmatrix} 
 b_{11} & 0 & b_{13} & 0\\
0& 0 & b_{23} &0\\
0 & 1 & b_{33} & 0\\
0 & 0 & b_{43} & b_{11}
 \end{pmatrix}.$$
 Furthermore, they preserve an alternating form, so for $i\in \{1,2\}$ we have $g_i^TXg_i=X$ for some invertible matrix
 $$ X=\begin{pmatrix} x_{11} & x_{12} & x_{13} & x_{14}\\
 -x_{12} & x_{22} & x_{23} & x_{24} \\
 -x_{13} & -x_{23} & x_{33} & x_{34} \\
   -x_{14} & -x_{24} & -x_{34} & x_{44} 
  \end{pmatrix}$$ 
  with $x_{jj}=0$.
  We may assume ${\rm det}(X)=1$. \\
  Let $R=\mathbb{Z}[\alpha_{13}, \alpha_{14},\alpha_{23},\alpha_{24}, \beta_{11},\beta_{13}, \beta_{23},\beta_{33},\beta_{43},\delta_{ij}: 1\leq i \leq j \leq 4]$
  where  $\alpha_{ij}$, $\beta_{ij}$ and $\delta_{ij}$ are indeterminates. 
  Set 
  $$ \Gamma_1=\begin{pmatrix} 0 & -1 & \alpha_{13} & \alpha_{14}\\
1& -1 & \alpha_{23} &\alpha_{24}\\
0 & 0 & 1 & 0\\
0 & 0 & 0 & 1
 \end{pmatrix}, \quad 
 \Gamma_2=\begin{pmatrix} 
 \beta_{11} & 0 & \beta_{13} & 0\\
0& 0 & \beta_{23} &0\\
0 & 1 & \beta_{33} & 0\\
0 & 0 & \beta_{43} & \beta_{11}
 \end{pmatrix},$$
 $$  \Delta=\begin{pmatrix} \delta_{11} & \delta_{12} & \delta_{13} & \delta_{14}\\
 -\delta_{12} & \delta_{22} & \delta_{23} & \delta_{24} \\
 -\delta_{13} & -\delta_{23} & \delta_{33} & \delta_{34} \\
   -\delta_{14} & -\delta_{24} & -\delta_{34} & \delta_{44} 
  \end{pmatrix},$$
  and let $I$ be the ideal of $R$ generated by ${\rm det}(\Delta)-1$ and the entries of the five matrices $\Gamma_1^3-I_4$, $\Gamma_2^3-I_4$, $\Gamma_1^T\Delta\Gamma_1-\Gamma_1$,
  $\Gamma_2^T\Delta\Gamma_2-\Gamma_2$ and $\Delta+\Delta^T$.  Then $\epsilon: R\rightarrow \mathbb{F}:$ $\alpha_{ij}\mapsto a_{ij}$, $\beta_{ij}\mapsto b_{ij}$, $\delta_{ij}\mapsto x_{ij}$ defines a ring homomorphism that factors over $R/I$. In particular if, with inclusively the notation from (\ref{e:charpoly}),  $$\sum_{w \in F_2} \lambda_w \chi_i(w(\Gamma_1,\Gamma_2))\in I,$$
  where $i\in \{0,1,2,3\}$, $\lambda_w \in R$ for all $w \in F_2$ and $\lambda_w=0$ for all but finitely many $w$, then 
   $$ \sum_{w\in F_2} \epsilon(\lambda_w) \chi_i(w(g_1,g_2))=0.$$
   Ideal membership can be tested using Gr\"{o}bner bases in MAGMA \cite{Magma} and the claims of the proposition are easily verified. We compute 
   \begin{eqnarray*}
   c_{1-2} &=&  -2t_{12}-1\\
   t_{121-2} & =&t_{12-12}= -t_{12}^2-4t_{12}-3\\
      t_{12-1-2} &= &t_{12}^2+4t_{12}+8\\
      t_{1-21-2}&=&t_{12}^2+12t_{12}+18\\
      t_{1-2} &=& -t_{12}-4\\
      c_{12}&=&2t_{12}+7.
   \end{eqnarray*}
   Also $t_{1-2}$ satisfies the required relation.  \\

  \textbf{Subcase (a.2).} Suppose that $g_{V/U}$ has minimal polynomial $(T-1)^2$. Then ${\rm char}(\mathbb{F})=3$ and so $g_1$  and $g_2$ have a unique eigenvalue, namely $1$. There are two cases to consider, namely the case where $B_1=(v,g_1v,g_2g_1v,g_1g_2g_1v)$ is a basis of $V$ and the case where $g_1g_2g_1v\in \langle v,g_1v,g_2g_1v\rangle$. In the former case, with respect to the basis $B_1$ of $V$, the elements $g_1$ and $g_2$ have the following form
   $$ g_1=\begin{pmatrix} 0 & -1 & 0& a_{14}\\
1& -1 & 0 &a_{24}\\
0 & 0 & 0 & -1\\
0 & 0 & 1 & -1
 \end{pmatrix}, \quad 
 g_2=\begin{pmatrix} 
 1 & 0 & b_{13} & b_{14}\\
0& 0 & b_{23} &b_{24}\\
0 & 1 & b_{33} & b_{34}\\
0 & 0 & b_{43} & b_{44}
 \end{pmatrix}.$$
 Using the method described above, we get:
    \begin{eqnarray*}
   c_{1-2} &=&  c_{12}+2t_{12}+1\\
   t_{121-2}  &=&t_{12-12}= -t_{12}^2+2c_{12}+t_{12}+1\\
      t_{12-1-2}& =& t_{12}^2+c_{12}+2t_{12}+1\\
      t_{1-21-2}&=&t_{12}^2-2c_{12}-2t_{12}-1\\
      t_{1-2} &= &-t_{12}-1,
   \end{eqnarray*}
   and $t_{1-2}$ satisfies the required relation. \\
   Suppose $g_1g_2g_1v \in \langle v, g_1v, g_2g_1v \rangle$. Since $\dim W=2$ and $v\in W\cap  \langle v, g_1v, g_2g_1v \rangle$, we deduce that 
   $$\dim  (W\cap  \langle v, g_1v, g_2g_1v \rangle)\in \{1,2\}.$$
   If  $$\dim  (W\cap  \langle v, g_1v, g_2g_1v\rangle)=2$$ then, as $w\not \in U$, we get that $\langle v, g_1v, g_2g_1v \rangle=\langle v, g_1v, w \rangle$ is $\langle g_1,g_2 \rangle$-invariant,  contradicting the irreducibility of $\langle g_1,g_2\rangle$. Hence  $$\dim  (W\cap  \langle v, g_1v, g_2g_1v \rangle)=1$$ and 
   $B_2=(v,g_1v,g_2g_1v,w)$ is a basis of $V$.  Since $g_{V/U}$ has minimal polynomial $(T-1)^2$, we deduce that with respect to the basis $B_2$ of $V$,  the elements $g_1$ and $g_2$ have the following form
$$ g_1=\begin{pmatrix} 0 & -1 & a_{13} & a_{14}\\
1& -1 & -a_{13} &a_{24}\\
0 & 0 & 1 & a_{34}\\
0 & 0 & 0 & 1
 \end{pmatrix}, \quad 
 g_2=\begin{pmatrix} 
 1 & 0 & b_{13} & 0\\
0& 0 & -1 &0\\
0 & 1 & -1 & 0\\
0 & 0 & b_{43} & 1
 \end{pmatrix}$$
 and we find
   \begin{eqnarray*}
   c_{1-2} &=&  -2t_{12}-1\\
   t_{121-2}  &=&t_{12-12}= -t_{12}^2-4t_{12}-3\\
      t_{12-1-2} &= &t_{12}^2+4t_{12}+8\\
      t_{1-21-2}&=&t_{12}^2+12t_{12}+18\\
      t_{1-2} &= &-t_{12}-4\\
      c_{12}&=&2t_{12}+7,
   \end{eqnarray*}
   and $t_{1-2}$ satisfies the required relation.  \\
   
\textbf{Case (b).} Assume now that $(v,g_1v, g_1^2v)$ is linearly independent for every $0\neq v \in W$. Fix such a $v$, and let $U=\langle v,g_1 v, g_1^2 v \rangle$ and let  $w\in W$ be such that $W=\langle  v,w \rangle$.  Note that $U$ is $\langle g_1\rangle$-invariant. We claim that $(v,g_1v,g_1^2v,w)$ is a basis of $V$. Suppose not. Then $w\in U$. Hence both $v+g_1v+g_1^2v$ and $w+g_1w+g_1^2 w$ belong to $U$ and are eigenvectors of $g_1$ with eigenvalue $1$. Since the eigenspace of $g_1$ on $U$ with eigenvalue 1 is one-dimensional, we get   $$v+g_1v+g_1^2v=\lambda (w+g_1w+g_1^2w)$$ for some nonzero $\lambda \in \mathbb{F}$. Set $u=v-\lambda w$. Then $u\in W$, $u\neq 0$ (as $v$ and $w$ are linearly independent) and $u+g_1u+g_1^2u=0$, contradicting our assumption. \\
With respect to the basis   $(v,g_1v,g_1^2v,w)$ of $V$, the elements $g_1$ and $g_2$ have the following form
$$ g_1=\begin{pmatrix} 0 & 0 & 1 & a_{14}\\
1& 0 & 0 &a_{24}\\
0 & 1 & 0 & a_{34}\\
0 & 0 & 0 & 1
 \end{pmatrix}, \quad 
 g_2=\begin{pmatrix} 
 b_{11} & b_{12} & b_{13} & 0\\
0& b_{22} &  b_{23} & 0\\
0 & b_{32} & b_{33} & 0\\
0  & b_{42} & b_{43} & b_{11}
 \end{pmatrix}.$$
Also with respect to the basis $(v,g_1v,g_1^2v,w-a_{24}v)$ of $V$, the elements $g_1$ and $g_2$ have the following form
$$ g_1=\begin{pmatrix} 0 & 0 & 1 & a_{14}\\
1& 0 & 0 &0\\
0 & 1 & 0 & a_{34}\\
0 & 0 & 0 & 1
 \end{pmatrix}, \quad 
 g_2=\begin{pmatrix} 
 b_{11} & b_{12} & b_{13} & 0\\
0& b_{22} &  b_{23} & 0\\
0 & b_{32} & b_{33} & 0\\
0  & b_{42} & b_{43} & b_{11}
 \end{pmatrix}.$$
 Note that either $g_2$ has characteristic polynomial $(T-1)^2(T^2+T+1)$ or  $g_2$ has characteristic polynomial $(T^2+T+1)^2$  and ${\rm char}(\mathbb{F})\neq 3$. 
 Note that the characteristic polynomial $p_{g_2}$ of $g_2$ satisfies  $$p_{g_2}(T)=(T-b_{11})^2(T^2-(b_{22}+b_{33})T+b_{22}b_{33}-b_{23}b_{32}).$$  
 In the first case, $b_{11}=1$, and so $-(b_{22}+b_{33})=1$. Hence the elements $g_1$ and $g_2$ reduce further to
  $$ g_1=\begin{pmatrix} 0 & 0 & 1 & a_{14}\\
1& 0 & 0 &0\\
0 & 1 & 0 & a_{34}\\
0 & 0 & 0 & 1
 \end{pmatrix}, \quad 
 g_2=\begin{pmatrix} 
 1 & b_{12} & b_{13} & 0\\
0& b_{22} &  b_{23} & 0\\
0 & b_{32} & -b_{22}-1 & 0\\
0  & b_{42} & b_{43} & 1
 \end{pmatrix}$$
and we find 
\begin{eqnarray*}
c_{1-2} &=&2t_{1-2}+c_{12}-2t_{12}\\
t_{121-2}&=&t_{12-12}=t_{1-2}t_{12}-t_{1-2}-c_{12}+t_{12}\\
t_{12-1-2}&=&-t_{1-2}t_{12}+4t_{1-2}-4t_{12}+4c_{12}-4\\
t_{1-21-2}&=&-2t_{1-2}t_{12}-t_{12}^2+6t_{1-2}+7c_{12}-4t_{12}-7\\
0&=& t_{1-2}^2+(2t_{12}-10)t_{1-2}+t_{12}^2-9c_{12}+8t_{12}+7,
\end{eqnarray*}
and $t_{1-2}$ satisfies the required relation.\\
In the second case, examining the coefficient of $T^3$ in $p_{g_2}$, we get $-(b_{33}+b_{22})-2b_{11}=2$, and so $b_{33}=-b_{22}-2b_{11}-2$. 
Hence the elements $g_1$ and $g_2$ reduce further to
  $$ g_1=\begin{pmatrix} 0 & 0 & 1 & a_{14}\\
1& 0 & 0 &0\\
0 & 1 & 0 & a_{34}\\
0 & 0 & 0 & 1
 \end{pmatrix}, \quad 
 g_2=\begin{pmatrix} 
 b_{11} & b_{12} & b_{13} & 0\\
0& b_{22} &  b_{23} & 0\\
0 & b_{32} & -2-2b_{11}-b_{22} & 0\\
0  & b_{42} & b_{43} & b_{11}
 \end{pmatrix}$$
and we find 
\begin{eqnarray*}
c_{1-2}& =&c_{12}+2t_{12}+1\\
t_{121-2}&=&-t_{12}^2+2c_{12}+t_{12}+1\\
t_{12-12}&=&-t_{12}^2-c_{12}-2t_{12}+1\\
t_{1-2-1-2}&=& t_{12}^2+c_{12}+2t_{12}+1\\
t_{1-21-2}&=& t_{12}^2-2c_{12}-2t_{12}-1\\
t_{1-2}&=&-t_{12}-1,
\end{eqnarray*}
and $t_{1-2}$ satisfies the required relation.
\end{proof}

\begin{lem}\label{l:chi}
Let $c$ be a positive integer and let $\zeta_c \in \mathbb{C}$ be a primitive $c$-th root of unity. For $i, j \in \mathbb{Z}/c\mathbb{Z}$, let $\Theta_c^{i,j}\in \mathbb{Z}[T]$ be the minimal polynomial of $\zeta_c^{i}+\zeta_c^{-i}+\zeta_c^{j}+\zeta_c^{-j}$, and let $\Delta_c^{i,j}\in \mathbb{Z}[T]$ be the minimal polynomial of $ \zeta_c^{i+j}+\zeta_c^{i-j}+\zeta_c^{-i+j}+\zeta_c^{-i-j}+2$. Set $$\Theta_c={\rm lcm}(\{\Theta_c^{i,j}: i, j \in  \mathbb{Z}/c\mathbb{Z} \}) \quad \textrm{and} \quad \Delta_c={\rm lcm}(\{\Delta_c^{i,j}: i, j \in  \mathbb{Z}/c\mathbb{Z} \}).$$
Then for every element  $g$ of ${\rm Sp}_4(\mathbb{F})$ of order dividing $c$, we have
$$\Theta_c(\chi_3(g))=0 \quad \textrm{and} \quad \Delta_c(\chi_2(g))=0.$$  
\end{lem}

\begin{proof}
Set $p={\rm char}(\mathbb{F})$. Let $g$ be an element of ${\rm Sp_4}(\mathbb{F})$ of order dividing $c$.  Write $c=df$ where $d$ and $f$ are positive integers such that $f\geq 1$ is the largest  nonnegative power of $p$ dividing $c$. In particular, $d$ and $f$ are coprime integers.  Let $\omega\in \mathbb{F}$ be a primitive $d$-th root of unity. There exist  $i,j\in \mathbb{Z}/c\mathbb{Z}$ such that 
$$p_g(T)=(T-\omega^{i})(T-\omega^{-i})(T-\omega^j)(T-\omega^{-j}).$$
In particular, by (\ref{e:charpoly}), $$\chi_3(g)=\omega^i+\omega^{-i}+\omega^j+\omega^{-j} \quad \textrm{and} \quad \chi_2(g)= 2+ \omega^{i+j}+\omega^{i-j}+\omega^{-i+j}+\omega^{-i-j}.$$
Let $R$ be the subring of $\mathbb{C}$ generated by $\zeta_c$ and let $\alpha: R\rightarrow \mathbb{F}$ be the ring homomorphism defined by $\alpha(\zeta_c)=\omega$.   The ring homomorphism $\alpha: R \rightarrow \mathbb{F}$ induces a ring homomorphism $\alpha: M_4(R)\rightarrow M_4(\mathbb{F})$ in a natural way (see Lemma \ref{l:mapchp}). Let $h \in M_4(R)$ and $g'\in M_4(\mathbb{F})$ be  defined as follows:
$$h={\rm diag}(\zeta_c^i, \zeta_c^{-i},\zeta_c^{j}, \zeta_c^{-j}) \quad \textrm{and} \quad g'={\rm diag}(\omega^i, \omega^{-i},\omega^{j}, \omega^{-j}).$$
Then $g'=\alpha(h)$ and $p_g(T)=p_{g'}(T)$ 
 Also $$\chi_3(h)= \zeta_c^i+\zeta_c^{-i}+\zeta^j+\zeta^{-j} \quad \textrm{and} \quad \chi_2(h)=2+\zeta^{i+j}+\zeta^{i-j}+\zeta^{-i+j}+\zeta^{-i-j}.$$
 It follows that  $$\Theta_c(\chi_3(h))=0 \quad \textrm{and} \quad \Delta_c(\chi_2(h))=0.$$  
 Note that for any $\Theta\in \mathbb{Z}[T]$ and any $r\in R$, we have $\Theta(\alpha(r))=\alpha(\Theta(r))$.
 Therefore $$\Theta_c(\alpha(\chi_3(h)))=\alpha(\Theta_c(\chi_3(h)))=0 \quad \textrm{and} \quad \Delta_c(\alpha(\chi_2(h)))=\alpha(\Delta_c(\chi_2(h)))=0.$$
 The result follows from the fact that   $\chi_k(g)=\chi_k(g')=\alpha(\chi_k(h))$ for $k\in\{1,2\}$. 
\end{proof}

We can now prove Theorem \ref{t:33sp}.

\noindent \textit{Proof of Theorem \ref{t:33sp}.}
We first consider part (i).
Suppose there is a positive integer $r$ such that ${\rm Sp}_4(p^r)$ is a  $(3,3,c)$-group. Let $\mathbb{F}$ be the finite field of characteristic $p$ and order $p^r$. Then there exist elements $g_1,g_2$ of ${\rm Sp}_4(\mathbb{F})$ of order 3 such that $g_1g_2$ has order  dividing $c$ and  $\langle g_1,g_2 \rangle ={\rm Sp}_4(\mathbb{F})$. Note that $\langle g_1,g_2\rangle$ is absolutely irreducible.  By Lemma \ref{l:chf}, $\mathbb{F}$ is the character field of $\langle g_1,g_2\rangle$. Let $\rho(X,Y,Z)=(Z+X+1)(Z^2+(2X-10)Z+X^2-9Y+8X+7)$.  By Proposition \ref{p:chf}, $\mathbb{F}$ is an image of $\mathbb{Z}[X,Y,Z]/\langle \rho(X,Y,Z)\rangle$. By Lemma \ref{l:chi}, this epimorphism factors  over 
$$ A=\mathbb{F}_p[X,Y,Z]/\langle \Theta_c(X), \Delta_c(Y), \rho(X,Y,Z) \rangle.$$  (The polynomials $\Theta_c, \Delta_c\in \mathbb{Z}[T]$ are defined  in Lemma \ref{l:chi}.) 
Since $A$ is a finite-dimensional algebra over the field $\mathbb{F}_p$, seen as a finite-dimensional vector space over $\mathbb{F}_p$,  $A$ has only finitely many subspaces, say $A_1,\dots, A_m$. In particular there exists $1\leq i\leq m$ such that $\mathbb{F}\cong A/A_i$.  Noting that $c$ has a finite number of divisors, it follows that there are only finitely many positive integers $r$ such that ${\rm Sp}_4(p^r)$ is a $(3,3,c)$-group. \\
We now consider part (ii). By part (i) it is enough to show that if ${\rm PSp}_4(p^r)$ is a $(3,3,c)$-group, then  ${\rm Sp}_4(p^r)$ is a $(3,3,2c)$-group. 
We suppose that $p>2$, as otherwise ${\rm PSp}_4(p^r)\cong {\rm Sp}_4(p^r)$.\\
 Suppose there is a positive integer $r$ such that  ${\rm PSp}_4(p^r)$ is a  $(3,3,c)$-group.  Let $\mathbb{F}$ be the finite field of characteristic $p$ and order $p^r$ and write $Z({\rm Sp}_4(\mathbb{F}))=Z$.  Then there exist elements $g_1,g_2$ of ${\rm Sp}_4(\mathbb{F})$ of order 3 such that $g_1g_2$ has order dividing $2c$ and  $\langle g_1,g_2 \rangle Z ={\rm Sp}_4(\mathbb{F})$. Since ${\rm Sp}_4(\mathbb{F})$ has no subgroup of index 2, we deduce that $\langle g_1,g_2\rangle={\rm Sp}_4(\mathbb{F})$ and so  ${\rm Sp}_4(\mathbb{F})$ is a $(3,3,2c)$-group.\\
 Part (iii) now follows from parts (i) and (ii) and Proposition \ref{p:lalalaboudou}. 
$\square$ \\


\begin{thebibliography}{99}

\bibitem{AB} M. Aschabher, R. Guralnick. Some applications of the first cohomology group. J. Algebra \textbf{90} (1984),  446Ð460.

\bibitem{Magma}W. Bosma, J. Cannon, C. Playoust. The Magma algebra system. I. The user language. J. Symbolic Comput. \textbf{24} (1997), 235--265.



\bibitem{Conder1980} M. Conder. Generators for alternating and symmetric groups. J. London Math. Soc.  \textbf{22} (1980), 75--86.

\bibitem{Conder1.5} M. Conder. Some results on quotients of triangle groups.
Bull. Austral. Math. Soc. \textbf{30} (1984),  73--90.

\bibitem{Conder} M. Conder. Hurwitz groups: a brief survey. Bull. Amer. Math. Soc. \textbf{23} (1990), 359--370. 

\bibitem{Conder2} M. Conder.  An update on Hurwitz groups. Groups Complex. Cryptol. \textbf{2} (2010),  35--49

\bibitem{Cooperstein} B. Cooperstein. Maximal subgroups of $G_2(2^n)$. J. Algebra \textbf{70} (1981), 23--36. 

\bibitem{Everitt} B. Everitt. Alternating quotients of Fuchsian groups. J. Algebra \textbf{223} (2000), 457--476.

\bibitem{Hurwitz} A. Hurwitz. Ueber algebraische Gebilde mit eindeutigen Transformationen in sich. Math. Ann. \textbf{41} (1893), 403--442. 


\bibitem{Katz} N. M. Katz. Rigid local systems (Princeton University Press, 1996).


\bibitem{Lang} S. Lang. Algebra. Graduate Texts in Mathematics \textbf{211}. Springer-Verlag, New York, 2002.


\bibitem{LLM1} M. Larsen, A. Lubotzky, C. Marion. Deformation theory and finite simple quotients of triangle groups I. J. Eur. Math. Soc. \textbf{16} (2014), 1349--1375.


\bibitem{LLM2} M. Larsen, A. Lubotzky, C. Marion. Deformation theory and finite simple quotients of triangle groups II. Groups Geom. Dyn.  \textbf{8} (2014), 811--836.


\bibitem{Lawther95} R. Lawther. Jordan block sizes of unipotent elements  in exceptional algebraic groups. Comm Algebra \textbf{23} (1995), 4125--4156.

\bibitem{Lawther} R. Lawther. Elements of specified order in simple algebraic groups. Trans. Amer. Math. Soc. \textbf{357} (2005), 221--245.  

\bibitem{Mariong2}  M. Liebeck,  A. Litterick, C. Marion.  A rigid triple of conjugacy classes in $G_2$. J. Group Theory \textbf{14} (2011), 31--36. 

\bibitem{Liebeck96} M. Liebeck, A. Shalev. Classical groups, probabilistic methods, and the $(2,3)$-generation problem. Ann. of Math. \textbf{144} (1996), 77--125. 

\bibitem{LSalt} M. Liebeck, A. Shalev. Fuchsian groups, coverings of Riemann surfaces, subgroup growth, random quotients and random
walks. J. Algebra \textbf{276} (2004), 552--601.

\bibitem{LiebeckSeitz} M. Liebeck, G. Seitz. Unipotent and nilpotent classes in simple algebraic groups and Lie algebras. Mathematical Surveys and Monographs \bf{180}. American Mathematical Society, Providence, RI, (2012).


\bibitem{Macbeath} A. Macbeath. Generators of the linear fractional groups. In Proc. Sympos.
Pure Math. vol XII (American Mathematical Society, 1969), pp. 14--32.


\bibitem{Marionconj} C. Marion. On triangle generation of finite groups of Lie type. J. Group Theory \textbf{13} (2010), 619--648. 

\bibitem{MarionLawther} C. Marion. Varieties of elements of given order in simple algebraic groups. Preprint available on the ArXiv.

\bibitem{Tamburini} L. Di Martino, M. Tamburini, A. Zalesskii. On Hurwitz groups of low rank. Comm. Algebra \textbf{28} (2000), 5383--5404.

\bibitem{MoodyPatera} R. Moody, J. Patera.  Characters of Elements of Finite Order in Lie Groups.  SIAM J. Algebraic Discrete Methods \textbf{5} (1984),  359--383.

\bibitem{Procesi} C. Procesi. The invariant theory of $n \times n$ matrices. Advances in Math. \textbf{19} (1976), 306--381.  





\bibitem{Tiep} P. Tiep, A. Zalesski. Real conjugacy classes in algebraic groups and finite groups of Lie type. J. Group Theory \textbf{8} (2005), 291 -- 315.

\bibitem{Weil} A. Weil. Remarks on the cohomology of groups. Ann. of Math. \textbf{80} (1964), 149--157.

\vspace{3cm}



\noindent Sebastian Jambor\\
\noindent seb.jambor@gmail.com\\

\noindent Alastair Litterick, Faculty of Mathematics, University of Bielefeld, Germany.\\
alitterick@math.uni-bielefeld.de\\

\noindent Claude Marion, Dipartimento di Matematica, Universit\`{a} degli Studi di Padova,  Padova, Italy.\\
\noindent marion@math.unipd.it  




\end{thebibliography}
\end{document}